\documentclass[11pt]{amsart}

\usepackage[applemac]{inputenc}
\usepackage{amsmath, amsthm, amssymb}
\usepackage[all]{xy}
\usepackage{marginnote}
\usepackage{color}

\theoremstyle{plain}
\newtheorem{theorem}{Theorem}[section]
\newtheorem*{theorem*}{Theorem}

\newtheorem{proposition}[theorem]{Proposition}

\newtheorem{lemma}[theorem]{Lemma}
\newtheorem{conjecture}[theorem]{Conjecture}
\newtheorem*{conjecture*}{Conjecture}

\theoremstyle{definition}
\newtheorem{definition}[theorem]{Definition}

\theoremstyle{remark}
\newtheorem{remark}[theorem]{Remark}
\newtheorem{claim}[theorem]{Claim}

\newcommand{\Ric}{\operatorname{Ric}}
\newcommand{\HSC}{\operatorname{HSC}}
\newcommand{\HBC}{\operatorname{HBC}}
\newcommand{\scal}{\operatorname{scal}}

\def\Xint#1{\mathchoice
{\XXint\displaystyle\textstyle{#1}}%
{\XXint\textstyle\scriptstyle{#1}}%
{\XXint\scriptstyle\scriptscriptstyle{#1}}%
{\XXint\scriptscriptstyle\scriptscriptstyle{#1}}%
\!\int}
\def\XXint#1#2#3{{\setbox0=\hbox{$#1{#2#3}{\int}$ }
\vcenter{\hbox{$#2#3$ }}\kern-.6\wd0}}

\def\dashint{\Xint-}

\begin{document}
\bibliographystyle{alpha}

\title[Negative curvature and positivity of the canonical calss]{Kobayashi hyperbolicity, negativity of the curvature and positivity of the canonical bundle} 

\author{Simone Diverio }
\address{Simone Diverio \\ Dipartimento di Matematica \lq\lq Guido Castelnuovo\rq\rq{} \\ SAPIENZA Università di Roma \\ Piazzale Aldo Moro 5 \\ I-00185 Roma.}
\email{diverio@mat.uniroma1.it}

\thanks{}
\subjclass[2010]{}
\date{\today}




\maketitle
\tableofcontents

\section{Introduction}

Let $X$ be a compact complex manifold. An \emph{entire curve} traced in $X$ is by definition a non constant holomorphic map $f\colon\mathbb C\to X$. By Brody's criterion $X$ is Kobayashi hyperbolic if and only if $X$ does not admit any entire curve.

At the very beginning of the theory, in the early 70's, very few examples of (higher dimensional) compact complex manifolds where known: mainly compact quotients of bounded domains in $\mathbb C^n$. 
Such quotients admits the Bergman metric, whose class lies by construction in the opposite of the first Chern class of $X$. Thus, for them, the canonical bundle is positive and therefore ample. In particular such quotients are projective.

We can guess that this lack of knowledge of examples on the one hand, and the positivity property of the canonical bundle of the known hyperbolic compact complex manifolds on the other hand, led S. Kobayashi to conjecture the following.

\begin{conjecture*}[Kobayashi '70]
Let $X$ be a compact Kähler (or projective) manifold which is Kobayashi hyperbolic. Then, $K_X$ is ample.
\end{conjecture*}

In the same vein, Kobayashi also asked in \cite{Kob74} whether a compact hyperbolic complex manifold has always infinite fundamental group. While the answer to this latter question is nowadays known to be negative (we know plenty of examples of simply connected compact complex manifold, mainly given by smooth projective general complete intersections of high degree), the former question is still widely open (and believed to be true).

Observe that, at least in the projective case, we now know since Mori's breakthrough \cite{Mor79} that being hyperbolic implies the nefness of the canonical bundle, due to the absence of rational curves. Thus, the canonical class is at least in the closure of the ample cone. What we want is then to show that hyperbolicity pushes the canonical class a little bit further into the ample cone.

Beside the class of smooth compact quotients of bounded domain, another remarkable class of compact hyperbolic manifolds ---known since the beginning of the theory--- is given by compact hermitian manifolds whose holomorphic sectional curvature is negative. Even if this is for sure an important class where to test conjectures on hyperbolic manifolds, it is somehow surprising that until very recently the Kobayashi conjecture was not known even for this class.

The principal aim of this chapter is to present in full detail a proof of the following statement due to Wu and Yau, which settles Kobayashi's conjecture for negatively curved hyperbolic projective manifolds.

\begin{theorem}[{\cite{WY16}}]
Let $X$ be a smooth projective manifold, and suppose that $X$ carries a Kähler metric $\omega$ whose holomorphic sectional curvature is everywhere negative. Then, $X$ posseses a (possibly different) Kähler metric $\omega'$ whose Ricci curvature is everywhere negative. In particular, $K_X$ is ample.
\end{theorem}

Observe that, since the holomorphic sectional curvature decreases when the metric is restricted to smooth submanifolds, as a direct consequence one obtains that every smooth submanifold of a compact Kähler manifold with negative holomorphic sectional curvature has ample canonical bundle. This observation goes in the direction of the celebrated Lang conjecture which predicts the following.

\begin{conjecture}[{\cite{Lan86}}]
Let $X$ be a smooth projective complex manifold. Then, $X$ is Kobayashi hyperbolic if and only if $X$ as well as all of its subvarieties are of general type.
\end{conjecture}

Thus, since a projective manifold with ample canonical bundle is of general type, Wu and Yau's theorem is also a confirmation in the negatively curved case, as long as only smooth subvarieties are concerned, of (one direction of) Lang's conjecture. It is therefore of primordial importance to extend their result in the singular case. The exact statement one should try to prove is the following.
Let $X$ be an irreducible projective variety and suppose to be able to embed $X$ into a projective manifold $Y$ supporting a Kähler metric whose holomorphic sectional curvature is negative, at least locally around $X$. Thus, there should exists a modification\footnote{By definition, a modification $\mu\colon\tilde X\to X$ is a proper surjective holomorphic map, such that there exists a proper analytic subset $S\subset X$ with the property that $\mu|_{\tilde X\setminus\mu^{-1}(S)}\colon\tilde X\setminus\mu^{-1}(S)\to X\setminus S$ is a biholomorphism.} $\mu\colon\tilde X\to X$, with $\tilde X$ smooth, such that the canonical bundle $K_{\tilde X}$ is big.

\subsubsection*{Addendum}

The Wu--Yau--Tosatti--Yang theorem (as well as its generalization by Diverio and Trapani) has been very recently proved by H. Guenancia in \cite{Gue18} also for singular subvarieties as mentioned here above. His very general result stems from a highly non trivial generalization of ideas explained in this chapter, involving also deep results from the Minimal Model Program. Its most general version can be stated as follows.

\begin{theorem}[Guenancia {\cite[Theorem B]{Gue18}}]
Let $(X,D)$ be a pair consisting of a projective manifold $X$ and a reduced divisor $D = \sum_{i\in I}D_i$ with simple normal crossings. Let $\omega$ be a K\"ahler metric on $X^\circ := X\setminus D$ such that there exists $\kappa_0 > 0$ satisfying
$$
\forall (x,v)\in X^\circ\times T_{X,x}\setminus\{0\},\quad\operatorname{HSC}_\omega(x,[v])<-\kappa_0.
$$
Then, the pair $(X, D)$ is of log general type, that is, $K_X + D$ is big. 

If additionally $\omega$ is assumed to be bounded near $D$, then $K_X$ is itself big.
\end{theorem}

This theorem thus provides a full confirmation of Lang's conjecture for compact Kähler manifolds with negative holomorphic sectional curvature.

Finally, Lang's conjecture has been settled very recently also in the particular case of compact free quotients of bounded domains (and in a slightly more general context, indeed) in \cite{DB18}.

\subsection{Organization of the chapter}

Beside the introduction, this chapter is made up of five sections. 

Section 1 is devoted to build the proper background in complex differential geometry in order to go into the proof of Wu--Yau's theorem, as well as the basic notions of complex hyperbolicity. In particular we summarize the different kinds of curvature in the riemannian setting as well as in the hermitian setting, putting in evidence the relations of between these notions in the Kähler case. Moreover, we explain if and how the sign of a particular notion of curvature propagates to others. We also take the opportunity to recall how the negativity of the holomorphic sectional curvature gives the Kobayashi hyperbolicity of a compact hermitian manifold.

Section 2 has a birational geometric flavor, and we try to motivate in this framework Kobayashi's conjecture as well as Wu--Yau's theorem and its generalizations using standard tools and conjectures. Namely, assuming the abundance conjecture and using the Iitaka fibration, we try to make clear how compact projective manifolds with trivial real first Chern class enter naturally into the picture and how to rule them out using the negativity (or even the quasi-negativity) of the holomorphic sectional curvature.

In Section 3 we present in full details an algebraic criterion due to J.-P. Demailly which a compact hermitian manifold must satisfy in order to have negative holomorphic sectional curvature. As a consequence, we construct (still following Demailly) an example of smooth projective surface which is hyperbolic, has ample canonical bundle, but nevertheless does not admit any negatively curved hermitian metric. This shows that Wu--Yau's theorem is unfortunately a confirmation of Kobayashi's conjecture only in a particular (although important) case.

Section 4 is the heart of the chapter and we present therein a complete, detailed proof of Wu--Yau's theorem. The proof is divided into several steps, in order to make the strategy more insightful. We tried to really work out every computation and estimate, perhaps even paying the price to be slightly redundant, to keep the chapter fully self-contained.

Finally, in Section 5 we present a couple of generalizations of Wu--Yau's theorem, namely the Kähler case due to V. Tosatti and X. Yang, and the (from a certain point of view, sharp) quasi-negative case due to S. Trapani and the author.

\subsubsection*{Acknowledgements} 
The author wish to express his gratitude to the anonymous referee for having read with great attention the present chapter, and for the uncountable valuables suggestions which really improved the exposition. 

\section{Complex differential geometric background and hyperbolicity}

The material in this section is somehow standard, but we take the opportunity here to fix notations and explain some remarkable facts which are not necessarily in everybody's background. We refer to \cite{Dem,Huy05,Zhe00} for an excellent and more systematic treatment of the subject.

Let $X$ be a complex manifold of complex dimension $n$, and let $h$ be a hermitian metric on its tangent space $T_X$, which is considered as a complex vector bundle endowed with the standard complex structure $J$ inherited from the holomorphic coordinates on $X$. Then, the real part $g$ of $h=g-i\omega$ defines a riemannian metric on the underlying real manifold, while its imaginary part $\omega$ defines a $2$-form on $X$. 

Now, one can consider both the riemannian or the hermitian theory on $X$. On the one hand we have the existence of a unique connection $\nabla$ on $T_X^\mathbb R$ ---the Levi--Civita connection--- which is both compatible with the metric $g$ and without torsion. Here the superscript $\mathbb R$ is put on $T_X$ to emphasize that we are looking at the real underlying manifold. We call the square of this connection $R=\nabla^2$, the riemannian curvature of $(T_X^\mathbb R,g)$. It is a $2$-form with values in the endomorphisms of $T_X^\mathbb R$. 

On the other hand, we can complexify $T_X$ and decompose it as a direct sum of the eigenbundles for the complexified complex structure $J\otimes\operatorname{Id}_\mathbb C$ relatives to the eigenvalues $\pm i$:
$$
T_X^\mathbb C=T_X\otimes\mathbb C\simeq T^{1,0}_X\oplus T^{0,1}_X.
$$
We have a natural vector bundle isomorphism
$$
\begin{aligned}
\xi\colon & T_X^\mathbb R\to T^{1,0}_X \\
& v\mapsto \frac 12(v-iJv)
\end{aligned}
$$
which is moreover $\mathbb C$-linear: $\xi\circ J=i\xi$. There is a natural way to define a hermitian metric on $T_X^\mathbb C$, as follows. We first consider the $\mathbb C$-bilinear extension $g^\mathbb C$ of $g$, and then its sesquilinear form $\tilde h$ made up using complex conjugation in $T_X^\mathbb C$:
$$
\tilde h(\bullet,\bullet):=g^\mathbb C (\bullet,\bar\bullet).
$$
Such a hermitian metric realizes the direct sum decomposition above as an orthogonal decomposition. The complexification of $\omega$, which we still call $\omega$ by an abuse of notation, is then a real positive $(1,1)$-form. These three notions, namely a hermitian metric on $T_X$, a hermitian metric on $T^{1,0}_X$, and a real positive $(1,1)$-form are essentially the same, since there is a canonical way to pass from one to the other.

Now, we know that there exists a unique connection $D$ on $T_X^{1,0}$ which is both compatible with $\tilde h$ and the complex structure: the Chern connection. We call the square of this connection $\Theta=D^2$, the Chern curvature of $(T_X^{1,0},\tilde h)$. It is a $(1,1)$-form with values in the anti-hermitian endomorphisms of $(T_X^{1,0},\tilde h)$.

A basic question is then: can we compare these two theories \textsl{via} $\xi$? The answer is classical and surprisingly simple. The riemannian theory and the hermitian one are the same if and only if the metric $h$ is Kähler, \textsl{i.e.} if and only if $d\omega=0$. In other words, the metric is Kähler if and only if
$$
D=\xi\circ\nabla\circ\xi^{-1},
$$
and of course, in this case, $\Theta=\xi\circ R\circ\xi^{-1}$ (see \textsl{e.g.} \cite[\S 4.A]{Huy05}).

\subsection{Notions of curvature in riemannian and hermitian geometry and their correlation in the Kähler case.}

We now give a brief overview of the different notions of curvature in the setting respectively of riemannian geometry and hermitian geometry. Then, we shall compare them in the case of Kähler metrics, with particular attention to the \lq\lq propagation\rq\rq{} of signs.

\subsubsection{The riemannian case}

Let $(M,g)$ be a riemannian manifold, $\nabla$ its Levi--Civita connection and $R$ its riemannian curvature. To this data it is attached the classical notion of \emph{sectional curvature} $\mathcal K_g$ of $g$. It is a function which assigns to each $2$-plane $\pi=\operatorname{Span}(v,w)$ in $T_M$ the real number
$$
\mathcal K_g(\pi)=-\frac{\langle R(v,w)\cdot v,w\rangle_g}{||v||^2_g||w||^2_g-|\langle v,w\rangle_g|^2}.
$$
One can verify that this function completely determines the riemannian curvature tensor. One usually also considers other \lq\lq easier\rq\rq{} tensors obtained by performing some type of contractions on $R$, for instance the \emph{Ricci curvature} $r_g$ and the \emph{scalar curvature} $s_g$. The former is a symmetric $2$-tensor defined by
$$
r_g(u,v)=\operatorname{tr}_{T_M}\bigl(w\mapsto R(w,u)\cdot v\bigr).
$$
The latter is the real function on $M$ obtained by taking the trace of the Ricci curvature with respect to $g$:
$$
s_g=\operatorname{tr}_{g}r_g.
$$
One can straightforwardly show that, up to a positive factor which depends only on the dimension of $M$, the Ricci curvature can be obtained as an average of sectional curvatures and the scalar curvature as an average of Ricci curvatures. In particular the sign of the sectional curvature \lq\lq dominates\rq\rq{} the sign of the Ricci curvature which, in turn, \lq\lq dominates\rq\rq{} the sign of the scalar curvature. 

\subsubsection{The hermitian case}

We now look at the complex case. We start more generally with the notion of \emph{Griffiths curvature} for a holomorphic hermitian vector bundle $(E,h)$ over a complex manifold $X$. In this situation, we also have a unique connection both compatible with the metric and the holomorphic structure on $E$, whose curvature we still denote by $\Theta$. 

It is a $(1,1)$-form with values in the anti-hermitian endomorphisms of $(E,h)$. The Griffiths curvature assigns to each pair $(v,\zeta)\in T_{X,x}^{1,0}\times E_x$, $x\in X$, the real number given by
$$
\theta_{E,h}(v,\zeta)=\langle\Theta(v,\bar v)\cdot\zeta,\zeta\rangle_h.
$$
It has the remarkable property (which is a special case of what is called more generally Griffiths formulae) that it decreases when passing to holomorphic subbundles. Suppose that $S\subseteq E$ is a holomorphic subbundle of $E$ and endow it with the restriction metric $h|_S$. Then, given $x\in X$, $\zeta\in S_x\subseteq E_x$, and $v\in T_{X,x}^{1,0}$, we always have
$$
\theta_{S,h|_S}(v,\zeta)\le\theta_{E,h}(v,\zeta).
$$
Now, we look more closely at the special case where $E=T^{1,0}_X$, and the hermitian metric is $\tilde h$ as above. In this case, the Griffiths curvature is nothing but (up to normalization) what is classically called \emph{holomorphic bisectional curvature}. To be more precise, given a point $x\in X$ and two non-zero holomorphic tangent vectors $v,w\in T^{1,0}_{X,x}\setminus\{0\}$ we define the holomorphic bisectional curvature in the directions given by $v,w$ as 
$$
\operatorname{HBC}_h(x,[v],[w])=\frac{\theta_{T_X^{1,0},\tilde h}(v,w)}{||v||^2_{\tilde h}||w||^2_{\tilde h}}.
$$
By a slight abuse of notation, we may possibly confuse and interchange $h$, $\tilde h$ and $\omega$ in what follows.

The \emph{holomorphic sectional curvature} is defined to be the restriction of the holomorphic bisectional curvature to the diagonal:
$$
\operatorname{HSC}_h(x,[v])=\operatorname{HBC}_h(x,[v],[v])=\frac{\theta_{T_X^{1,0},\tilde h}(v,v)}{||v||^4_{\tilde h}}.
$$

In the spirit of the riemannian case, we can construct a closed real $(1,1)$-form, the \emph{Chern--Ricci form}, by taking the trace with respect to the endomorphism part of the Chern curvature. We also normalize it in such a way that its cohomology class coincides with the first Chern class of the manifold, namely:
$$
\operatorname{Ric}_\omega=\frac i{2\pi}\operatorname{tr}_{T^{1,0}_X}\Theta.
$$
The new feature here is that the Chern--Ricci tensor is a $2$-form always belonging to a fixed cohomology class, the first Chern class of $X$, independently of the choice of the metric. This is because, in general, the trace of the curvature of a vector bundle is the curvature of the induced connection on the determinant bundle, which in this case is the dual of the canonical bundle of $X$. By taking again the trace, but this time with respect to $\omega$, we get what is called the \emph{Chern scalar curvature}. Thus, it is by definition the unique real function $\operatorname{scal}_\omega\colon X\to\mathbb R$ such that
$$
\operatorname{Ric}_\omega\wedge\frac{\omega^{n-1}}{(n-1)!}=\operatorname{scal}_\omega\,\frac{\omega^{n}}{n!}.
$$

\subsubsection{Negativity of the holomorphic sectional curvature and hyperbolicity}

Before going further and explore ---when the metric is Kähler--- the links among the different notions of curvature we introduced in the riemannian and hermitian setting, we would like to explain here how the negativity of the holomorphic sectional curvature implies the Kobayashi hyperbolicity of the manifold. We want to do it here before entering the Kähler world, since this is a purely hermitian fact. Before proceeding further, let us recall the very basic definitions and notions about Kobayashi hyperbolicity (for an exhaustive treatment we refer to \cite{Kob05,Kob98}).

\medskip

Let $X$ be a complex space. We call a \emph{holomorphic disc} in $X$ a holomorphic map from the complex unit disc $\Delta$ to $X$. Given two points $p,q\in X$, consider a \emph{chain of holomorphic discs} from $p$ to $q$, that is a chain of points $p=p_0,p_1,\dots,p_k=q$ of $X$, pairs of point $a_1,b_1,\dots,a_k,b_k$ of $\Delta$ and holomorphic maps $f_1,\dots,f_k\colon\Delta\to X$ such that
$$
f_i(a_i)=p_{i-1},\quad f_i(b_i)=p_i,\quad i=1,\dots,k.
$$
Denoting this chain by $\alpha$, define its length $\ell(\alpha)$ by
$$
\ell(\alpha)=\rho(a_1,b_1)+\cdots+\rho(a_k,b_k)
$$ 
and a pseudodistance $d_X$ on $X$ by 
$$
d_X(p,q)=\inf_{\alpha}\ell(\alpha).
$$
This is the \emph{Kobayashi pseudodistance} of $X$.

\begin{definition}
The complex space $X$ is said to be \emph{Kobayashi hyperbolic} if the pseudodistance $d_X$ is actually a distance.
\end{definition}

For $\Delta$ the complex unit disc, it is easy to see using the usual Schwarz--Pick lemma\footnote{We recall here that the usual Schwarz-Pick lemma says that for $f\colon\Delta\to\Delta$ a holomorphic map, one has the following inequality:
$$
\frac{|f'(\zeta)|}{1-|f(\zeta)|^2}\le\frac{1}{1-|\zeta|^2}.
$$
This means exactly that holomorphic maps contract the Poincaré metric.} in one direction and the identity transformation in the other that $d_X=\rho$. Then $\Delta$ is hyperbolic. The entire complex plane is not hyperbolic: indeed the Kobayashi pseudodistance is identically zero. To see this from the very definition, take any two point $z_1,z_2\in\mathbb C$ and consider a sequence of holomorphic discs 
$$
\begin{aligned}
f_j\colon & \Delta\to\mathbb C \\
& \zeta\to z_1+j\zeta(z_2-z_1).
\end{aligned}
$$
It is important to remark here that the non hyperbolicity of the complex plane is connected to the possibility of taking larger and larger discs in $\mathbb C$. 

It is immediate to check that the Kobayashi pseudodistance has the fundamental property of being contracted by holomorphic maps: given two complex spaces $X$ and $Y$ and a holomorphic map $f\colon X\to Y$ one has for every pair of point $x,y$ in $X$
$$
d_Y(f(x),f(y))\le d_X(x,y).
$$ 
In particular, a Kobayashi hyperbolic complex space cannot contain any complex subspace which is not Kobayashi hyperbolic. 

\begin{remark}In other words, begin Kobayashi hyperbolic is an hereditary property for complex subspaces, analogously to what happens for the negativity of holomorphic sectional curvature, as prescribed by Griffiths' formulae. 
\end{remark}

The distance decreasing property together with the fact that the Kobayashi pseudodistance is identically zero on $\mathbb C$, implies immediately.

\begin{proposition}\label{distdecr}
If $X$ is a hyperbolic complex space, then every holomorphic map $f\colon\mathbb C\to X$ is constant.
\end{proposition} 

Let us now come at the infinitesimal analogue of the Kobayashi pseudodistance introduced above. For simplicity, we shall suppose that $X$ is a \emph{smooth} complex manifold but most of the things would work on an arbitrary singular complex space.

So, fix an arbitrary holomorphic tangent vector $v\in T_{X,x_0}$, $x_0\in X$: we want to give it an intrinsic length. Thus, define
$$
\mathbf{k}_{X}(v)=\inf\{\lambda>0\mid\exists f\colon\Delta\to X,\,f(0)=x_0,\,\lambda f'(0)=v\},
$$
where $f\colon\Delta\to X$ is holomorphic. Even with this infinitesimal form, it is straightforward to check that holomorphic maps between complex manifolds contract it and that in the case of the complex unit disc, it agrees with the Poincaré metric.

\begin{definition}
Let $X$ be a complex manifold and $\omega$ an arbitrary hermitian metric on $X$. We say that $X$ is \emph{infinitesimally Kobayashi hyperbolic} if $\mathbf{k}_{X}$ is positive definite on each fiber and satisfies a uniform lower bound
$$
\mathbf{k}_{X}(v)\ge\varepsilon ||v||_{\omega}
$$ 
when $v\in T_{X,x}$ and $x\in X$ describes a compact subset of $X$.
\end{definition}

The Kobayashi pseudodistance is the integrated form of the corresponding infinitesimal pseudometric (this is due to Royden).

\begin{theorem}[{\cite[(3.5.31) Theorem]{Kob98}}]
Let $X$ be a complex manifold. Then
$$
d_X(p,q)=\inf_\gamma\int_\gamma\mathbf{k}_{X}(\gamma'(t))\,dt,
$$
where the infimum is taken over all piecewise smooth curves joining $p$ to $q$.
\end{theorem}

In particular, if $X$ is infinitesimally hyperbolic, then it is hyperbolic.
Next theorem, which is due to Brody, is the simplest and most useful criterion for hyperbolicity. It gives a converse of Proposition \ref{distdecr} in the case where the target $X$ is compact. Fix any hermitian metric $\omega$ on the compact complex manifold $X$; we say that a holomorphic map $f\colon\mathbb C\to X$ is an \emph{entire curve} if it is non constant and that it is a \emph{Brody curve} if it is an entire curve with bounded derivative with respect to $\omega$ (or, of course, any other hermitian metric).

\begin{theorem}[{\cite{Bro78}, see \cite[(3.6.3) Theorem]{Kob98} for a more modern account}]\label{Brody}
Let $X$ be a compact complex manifold. If $X$ is not (infinitesimally) hyperbolic then there exists a Brody curve in $X$.
\end{theorem}

A first direct consequence of this theorem is that in the compact case, hyperbolicity and infinitesimal hyperbolicity are equivalent, since if $X$ is not infinitesimally hyperbolic then there exists an entire curve in $X$ and then two distinct points on this curve will have zero distance with respect to $d_X$. 
For more information on the localization of such a curve, we refer the reader to the remarkable results of \cite{Duv08}, which are also described in J. Duval's chapter of the present monograph.

The absence of entire holomorphic curves in a given complex manifold is often referred to as \emph{Brody hyperbolicity}. Thus, in the compact case, Brody hyperbolicity and Kobayashi hyperbolicity do coincide.

\medskip

We now come back to the fact that negativity of the holomorphic sectional curvature implies the Kobayashi hyperbolicity of the manifold. For our purposes, it is sufficient to deal with the smooth compact case even if more general statements can be established.

\begin{theorem}[{\cite[Theorem 4.1]{Kob05}}]
Let $(X,\omega)$ be a compact hermitian manifold such that we have $\operatorname{HSC}_\omega<0$. Then, $X$ is Kobayashi hyperbolic.
\end{theorem}

In the next section we will see that the negativity of the holomorphic sectional curvature is a sufficient but not necessary condition for Kobayashi hyperbolicity. The following proof is somehow slightly different (more formally than substantially) from the several others that can be found in literature. Let us highlight in particular that Brody's theorem is not really needed for the proof and can be for instance replaced by the Ahlfors--Schwarz lemma.

\begin{proof}
By Theorem \ref{Brody}, it is sufficient to show that every holomorphic map $f\colon\mathbb C\to X$ whose derivative is $\omega$-bounded is constant. So let $f$ be such a map a consider the function
$$
\begin{aligned}
F\colon & \mathbb C\to\mathbb R\cup\{-\infty\} \\
& t\mapsto \log||f'(t)||^2_\omega,
\end{aligned}
$$
which is clearly upper semi-continuous and bounded from above. Suppose by contradiction that $F$ is not identically $-\infty$, which corresponds to the fact that $f$ is not constant. Then, of course, the locus where $\log||f'(t)||^2_\omega$ is $-\infty$ is a discrete set. We now check that $\log||f'(t)||^2_\omega$ is a subharmonic function on the whole $\mathbb C$, which is moreover strictly subharmonic over $\{f'\ne 0\}$. Since any bounded subharmonic function on $\mathbb C$ is constant \cite[Proposition 2.7.3]{Kli91}, this gives a contradiction, because a constant function cannot be strictly subharmonic somewhere.

First of all we show the subharmonicity of $\log||f'||^2_\omega$, by showing that for all positive integer $k$ the smooth  functions defined on the whole complex plane $\psi_\varepsilon=\log(||f'||^2_\omega+\varepsilon)$ are subharmonic, \textsl{i.e.} $i\partial\bar\partial\psi_\varepsilon\ge 0$. For, since
$$
\log||f'||^2_\omega=\lim_{\varepsilon\to 0}\psi_\varepsilon
$$
pointwise, and the sequence $\{\psi_\varepsilon\}$ is decreasing, then the subharmonicity of $\log||f'||^2_\omega$ follows from \cite[Theorem 2.6.1, (ii)]{Kli91}. 

So, fix $\varepsilon>0$ and consider a point $t_0\in\mathbb C$. Call $x_0=f(t_0)\in X$ and choose holomorphic coordinates $(z_1,\dots,z_n)$ for $X$ centered at $x_0$ so that $x_0$ corresponds to $z=0$, and write $f=(f_1,\dots,f_n)$ for $f$ in these coordinates. Moreover, chose a normal coordinate frame $\{e_1,\dots,e_n\}$ for $(T_X,\omega)$ at $x_0$ \cite[(12.10) Proposition]{Dem}. With this choice we have that
$$
\langle e_l(z),e_m(z)\rangle_\omega=\delta_{lm}-\sum_{j,k=1}^nc_{jklm}\,z_j\bar z_k +O(|z|^3),
$$
where the $c_{jklm}$'s are the coefficients of the Chern curvature of $\omega$. Observe that, since the metric is not supposed to be Kähler, we can \textsl{a priori} not chose holomorphic coordinates around $x_0$ such that the $e_l$'s can be taken simply to be $\partial/\partial z_l$; nevertheless, by a constant change of coordinates, we can suppose that $e_l(x_0)$ equals $\partial/\partial z_l(x_0)$, at least at $x_0$. Now, of course there exist holomorphic functions $\varphi_j$, $j=1,\dots,n$, defined on a neighborhood of $t_0$ such that 
$$
f'(t)=\sum_{j=1}^n \varphi_j(t)\,e_j(t),
$$
so that around $t_0$ we have
$$
||f'(t)||^2_\omega=|\varphi(t)|^2-\sum_{j,k,l,m=1}^nc_{jklm}\, f_j(t)\overline{f_k(t)}\varphi_l(t)\overline{\varphi_m(t)}+O(|f|^3).
$$
Moreover, we have $f'_j(t_0)=\varphi_j(t_0)$ for all $j$, since the $e_l$'s and the $\partial/\partial z_l$'s agree at $t_0$.

Remark that, since $X$ is compact, there exists a positive constant $\kappa$ such that $\HSC_\omega<-\kappa$. This condition reads in our coordinates
$$
\sum_{j,k,l,m=1}^n c_{jklm}\, v_j\bar v_k v_l\bar v_m< -\kappa |v|^4, \quad\forall v=(v_1,\dots, v_n)\in \mathbb C^n.
$$
Now, we have to compute $\partial||f'||^2_\omega$, $\bar\partial||f'||^2_\omega$ and $\partial\bar\partial||f'||^2_\omega$ at $t_0$, since
$$
\partial\bar\partial\psi_\varepsilon=\frac{-1}{(||f'||^2_\omega+\varepsilon)^2}\partial||f'||^2_\omega\wedge\bar\partial||f'||^2_\omega+\frac 1{||f'||^2_\omega+\varepsilon}\partial\bar\partial||f'||^2_\omega.
$$
We find
$$
\begin{aligned}
&\partial||f'||^2_\omega|_{t_0}=\langle \varphi'(t_0),f'(t_0)\rangle \,dt, \\
&\bar\partial||f'||^2_\omega|_{t_0}=\langle f'(t_0),\varphi'(t_0)\rangle\,d\bar t, \\
&i\partial\bar\partial||f'||^2_\omega|_{t_0}  =i\biggl(|\varphi'(t_0)|^2-\sum_{j,k,l,m=1}^nc_{jklm}\,f'_j(t_0)\overline{f'_k(t_0)}f'_l(t_0)\overline{f'_m(t_0)}\biggr)\,dt\wedge d\bar t \\
&\qquad\qquad\qquad >i\bigl(|\varphi'(t_0)|^2+\kappa |f'(t_0)|^4\bigr)\,dt\wedge d\bar t,
\end{aligned}
$$
where the brackets just mean the standard hermitian product in $\mathbb C^n$. Putting all this together we obtain
$$
\begin{aligned}
i\partial\bar\partial\psi_\varepsilon|_{t_0} &>i\biggl( \frac{-|\langle f'(t_0),\varphi'(t_0) \rangle|^2}{(|f'(t_0)|^2+\varepsilon)^2}+\frac{|\varphi'(t_0)|^2+\kappa |f'(t_0)|^4}{|f'(t_0)|^2+\varepsilon}\biggr)\,dt\wedge d\bar t \\
&\ge i\biggl( \frac{-|f'(t_0)|^2|\varphi'(t_0)|^2}{(|f'(t_0)|^2+\varepsilon)^2}+\frac{|\varphi'(t_0)|^2+\kappa |f'(t_0)|^4}{|f'(t_0)|^2+\varepsilon}\biggr)\,dt\wedge d\bar t \\
&=i\,\frac{\kappa|f'(t_0)|^6+\varepsilon\bigl(|\varphi'(t_0)|^2+\kappa|f'(t_0)|^4\big)}{(|f'(t_0)|^2+\varepsilon)^2}\,dt\wedge d\bar t \ge 0,
\end{aligned}
$$
where we used  Cauchy--Schwarz for the second inequality, and so $\psi_\varepsilon$ is subharmonic at each point.

To conclude, observe that the very same computation with $\varepsilon=0$ makes sense away from points where $f'=0$, and give moreover strict positivity for $i\partial\bar\partial\log ||f'||^2_\omega$ at these points. Which means that, away from $\{f'=0\}$, $\log ||f'||^2_\omega$ is strictly subharmonic, as desired.
\end{proof}

\subsubsection{The Kähler case}

Suppose now that $(X,\omega)$ is a Kähler manifold, with $\omega=-\Im h$. Let $(X,g=\Re h)$ be the underlying riemannian manifold. We saw that this is precisely the case when $\Theta$ and $R$ correspond to each other via $\xi$. In this setting, using $\xi$, one can show easily the following useful relation between the holomorphic bisectional curvature and the riemannian sectional curvature. Let $v,w\in T_{X,x}$ be two independent (real) tangent vectors. Then,
\begin{multline*}
\operatorname{HBC}_\omega(x,[\xi(v)],[\xi(w)])=\frac{||Jv||^2_g||w||^2_g-|\langle Jv,w\rangle_g|^2}{||v||^2_g||w||^2}\,\mathcal K_g\bigl(\operatorname{Span}\{Jv,w\}\bigr)\\
+\frac{||v||^2_g||w||^2_g-|\langle v,w\rangle_g|^2}{||v||^2_g||w||^2}\,\mathcal K_g\bigl(\operatorname{Span}\{v,w\}\bigr).
\end{multline*}
In particular, if $\mathcal K_g$ has a sign at a certain point, so does $\operatorname{HBC}_\omega$, and the sign is the same. Moreover, by specializing at the diagonal, we get
$$
\operatorname{HSC}_\omega(x,[\xi(v)])=\mathcal K_g\bigl(\operatorname{Span}\{Jv,v\}\bigr),
$$
that is, the holomorphic sectional curvature is nothing but the riemannian sectional curvature computed on complex $2$-planes.

In the Kähler setting, not surprisingly, the riemannian Ricci curvature and the Chern--Ricci forms as well as the corresponding scalar curvatures are related to each other. The precise relation is as follows, for $v,w\in T_{X,x}$:
$$
\operatorname{Ric}_\omega\bigl(\xi(v),\overline{\xi(w)}\bigr)=\frac i{4\pi}\bigl(r_g(v,w)-ir_g(Jv,w)\bigr).
$$
In particular, since for $g$ the real part of a Kähler metric the Ricci tensor is $J$-invariant, we get that
$$
-i\operatorname{Ric}_\omega\bigl(\xi(v),\overline{\xi(v)}\bigr)=\frac 1{4\pi}r_g(v,v),
$$
that is $\operatorname{Ric}_\omega$ is a positive (resp. semi-positive, negative, semi-negative) $(1,1)$-form if and only if $r_g$ is a positive (resp. semi-positive, negative, semi-negative) symmetric bilinear form (observe that $-i\operatorname{Ric}_\omega\bigl(\bullet,\bar\bullet\bigr)$ is nothing but the real quadratic form associated to the real $(1,1)$-form $\operatorname{Ric}_\omega$). From this we also infer that
$$
\operatorname{scal}_\omega=\frac 1{4\pi}\,s_g.
$$
\begin{remark}\label{totscalcurv}
In the compact Kähler case, since both $\omega$ and $\operatorname{Ric}_\omega$ are closed forms, the total scalar curvature becomes a cohomological invariant, namely
$$
\begin{aligned}
\int_X s_g\,dV_g &=4\pi\int_X\operatorname{scal}_\omega\frac{\omega^n}{n!}\\
&=4\pi\int_X\operatorname{Ric}_\omega\wedge\frac{\omega^{n-1}}{(n-1)!}\\
&=\frac{4\pi}{(n-1)!}\,c_1(X)\cdot[\omega]^{n-1}.
\end{aligned}
$$
In particular, the total scalar curvature of a compact Kähler manifold with vanishing first Chern class must always be zero. This observation will be useful later.
\end{remark}

Now we explain, still in the Kähler case,  the link between the signs respectively of the holomorphic bisectional curvature and Ricci curvature, and of the holomorphic sectional curvature with the scalar curvature.

\begin{proposition}\label{average}
Let $(X,\omega)$ be compact Kähler manifold of complex dimension $n$, $x_0\in X$, and $v\in T^{1,0}_{X,x_0}\setminus\{0\}$. Then, we have
$$
-\frac{2\pi i}{n||v||^2_{\tilde h}}\operatorname{Ric}_\omega(v,\bar v)=\dashint_{S^{2n-1}} \HBC_\omega([v],[w])\,d\sigma(w),
$$
and
$$
\frac{4\pi}{n(n+1)}\scal_\omega(x_0)=\dashint_{S^{2n-1}}\HSC_\omega([v])\,d\sigma(v),
$$
where $d\sigma$ is the Lebesgue measure on the $\tilde h$-unit sphere $S^{2n-1}$ in $T^{1,0}_{X,x_0}$, and by $\dashint$ we mean taking the average, \textsl{i.e.}
$$
\dashint_{S^{2n-1}}=\frac{(n-1)!}{2\pi^n}\int_{S^{2n-1}}.
$$
\end{proposition}

In particular, we find that the sign of the holomorphic bisectional curvature of a Kähler metric dominates the sign of the Ricci curvature, while the sign of the holomorphic sectional curvature dominates the sign of the scalar curvature. 

On the other hand, recall that the holomorphic sectional curvature of a Kähler metric \emph{completely} determines the curvature tensor \cite[Lemma 7.19]{Zhe00}, so that the point is really whether and how it does spread the sign.

\begin{proof}
Fix holomorphic coordinates around $x_0$ such that 
$$
\omega(x_0)=i\sum_{j=1}^ndz_j\wedge d\bar z_j,
$$ 
\textsl{i.e.} such that $\{\partial/\partial z_j\}_{j=1,\dots,n}$ is a $\tilde h$-unitary basis at $x_0$. Now write the Chern curvature tensor at $x_0$ in these coordinates, to get
$$
\Theta\bigl(T^{1,0}_X,\tilde h\bigr)_{x_0}=\sum_{j,k,l,m=1}^n c_{jklm}\,dz_j\wedge dz_k\otimes\biggl(\frac{\partial}{\partial z_l}\biggr)^*\otimes\frac{\partial}{\partial z_m}.
$$
Since $\omega$ is Kähler, and we have chosen a unitary basis at $x_0$, we have the well known Kähler symmetries for the coefficients $c_{jklm}$'s of the curvature:
$$
c_{jklm}=c_{lmjk}=c_{lkjm}=c_{jmlk}=\overline{c_{kjml}}.
$$
Now we express the holomorphic bisectional curvature in coordinates, to get
$$
\HBC_\omega([v],[w])=\frac 1{|v|^2|w|^2}\sum_{j,k,l,m=1}^nc_{jklm}\,v_j\bar v_k w_l\bar w_m,
$$
where $v=\sum_j v_j\,\partial/\partial z_j$ and $w=\sum_j w_j\,\partial/\partial z_j$.
In particular, 
$$
\HSC_\omega([v])=\frac 1{|v|^4}\sum_{j,k,l,m=1}^nc_{jklm}\,v_j\bar v_k v_l\bar v_m.
$$
Moreover, we get for the Chern--Ricci curvature and for the Chern--scalar curvature the following expressions:
$$
\Ric_\omega=\frac i{2\pi}\sum_{j,k,l=1}^n c_{jkll}\,dz_j\wedge d\bar z_k,
$$
and
$$
\scal_\omega=\frac 1{2\pi}\sum_{j,l}c_{jjll}.
$$
Now, we compute
$$
\dashint_{S^{2n-1}} \HBC_\omega([v],[w])\,d\sigma(w) = \frac 1{|v|^2}\sum_{j,k,l,m=1}^nc_{jklm}\,v_j\bar v_k \dashint_{S^{2n-1}} w_l\bar w_m\,d\sigma(w).
$$
The integral on the right hand side is easily seen to be zero unless $l=m$, in which case gives $1/n$. Thus, we obtain
$$
\begin{aligned}
\dashint_{S^{2n-1}} \HBC_\omega([v],[w])\,d\sigma(w) &=\frac 1{n|v|^2}\sum_{j,k,l=1}^nc_{jkll}\,v_j\bar v_k \\
&=-\frac{2\pi i}{n||v||^2_{\tilde h}}\operatorname{Ric}_\omega(v,\bar v).
\end{aligned}
$$
For the holomorphic sectional curvature we have instead
$$
\dashint_{S^{2n-1}} \HSC_\omega([v])\,d\sigma(v) = \sum_{j,k,l,m=1}^nc_{jklm}\dashint_{S^{2n-1}}v_j\bar v_k  v_l\bar v_m\,d\sigma(v).
$$
Once again, it is easy to see that the integrals on the right hand side must be zero unless $j=k$ and $l=m$, or $j=m$ and $k=l$. We have (see for instance \cite{Ber66} or \cite[Lemma 2.2]{Div16} for a detailed and more general computation):
$$
\dashint_{S^{2n-1}}|v_j|^2|v_k|^2\,d\sigma(v)=
\begin{cases}
\frac 2{n(n+1)} & \textrm{if $j=k$} \\
\frac 1{n(n+1)} & \textrm{otherwise},
\end{cases}
$$
so that
$$
\begin{aligned}
\sum_{j,k,l,m=1}^nc_{jklm}\dashint_{S^{2n-1}}v_j\bar v_k  v_l\bar v_m\,d\sigma(v) &=
\sum_{j,l=1}^n c_{jjll}\dashint_{S^{2n-1}}|v_j|^2|v_l|^2\,d\sigma(v) \\
&\qquad+\sum_{j,k=1}^n c_{jkkj}\dashint_{S^{2n-1}}|v_j|^2|v_k|^2\,d\sigma(v) \\
&\qquad\qquad\qquad-\sum_{j=1}^n c_{jjjj}\dashint_{S^{2n-1}}|v_j|^4\,d\sigma(v)\\
&= 2\sum_{j,l=1}^n c_{jjll}\dashint_{S^{2n-1}}|v_j|^2|v_l|^2\,d\sigma(v) \\
&\qquad-\sum_{j=1}^n c_{jjjj}\dashint_{S^{2n-1}}|v_j|^4\,d\sigma(v)\\
&=2\sum_{j\ne l} c_{jjll}\dashint_{S^{2n-1}}|v_j|^2|v_l|^2\,d\sigma(v)\\
&\qquad+\sum_{j=1}^n c_{jjjj}\dashint_{S^{2n-1}}|v_j|^4\,d\sigma(v)\\
&= \frac 2{n(n+1)}\sum_{j,l=1}^n c_{jjll} \\
&=\frac{4\pi}{n(n+1)}\,\scal_\omega(x_0),
\end{aligned}
$$
where, for the second equality, we have used the Kähler symmetry $c_{jkkj}=c_{jjkk}$.
\end{proof}

The situation can be therefore summarized as follows:
$$
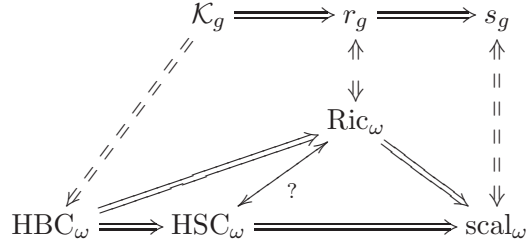
\xymatrix{
 & \mathcal K_{g} \ar@{=>}[r]\ar@{==>}[ddl] & r_g \ar@{=>}[r] & s_g  \\
  & & \Ric_\omega  \ar@{<==>}[u] \ar@{=>}[dr] &  \\
\HBC_\omega \ar@{=>}[r] \ar@{=>}[urr] & \HSC_\omega \ar@{=>}[rr] \ar@{<->}[ur]_{?} & & \scal_\omega \ar@{<==>}[uu]}
$$
The arrows $\Rightarrow$ in the diagram mean that the positivity (resp. semi-positivity, negativity, semi-negativity) of the source curvature implies the positivity (resp. semi-positivity, negativity, semi-negativity) of the target curvature. These arrows are always valid, even in the non Kähler setting. On the other hand, the dashed arrows are valid in the Kähler case only. It is however \textsl{a priori} unclear if and how the sign of the holomorphic sectional curvature propagates and determines the signs of the Ricci curvature. 

So, Yau's conjecture deals exactly with this issue, at least in the case of negativity: if $(X,\omega)$ is a compact Kähler manifold such that $\HSC_\omega<0$, then there exists a (possibly different) Kähler metric $\omega'$ on $X$ such that $\Ric_{\omega'}<0$. Note that, in particular, this implies that $K_X$ is ample and therefore, by Kodaira's embedding theorem, that $X$ is a projective algebraic manifold.

There is something subtle here which is going on, that tells us that Yau's conjecture is really something which is proper to negative curvature and cannot be extended to the positively curved world. Indeed, Hitchin \cite{Hit75} proved that the Hirzebruch surfaces $\Sigma_n=\mathbb P\bigl(\mathcal O_{\mathbb P^1}\oplus\mathcal O_{\mathbb P^1}(-n)\bigr)$, $n\in\mathbb N$, can always be endowed with a metric (of Hodge type) of positive holomorphic sectional curvature. But it is well known that the anticanonical bundle $-K_{\Sigma_n}$ is not ample if $n>1$. In particular, if $n>1$, $\Sigma_n$ cannot carry any Hermitian metric whose Chern--Ricci curvature is positive (we thank V. Tosatti for having brought back to our mind Hitchin's example).

\begin{remark}
It is somehow embarrassing, but we don't dispose ---at our best knowledge--- any example of a compact Kähler manifold $(X,\omega)$ such that $\HSC_\omega$ is negative but $\HBC_\omega$ or $\Ric_\omega$ do not have a sign. It would be of course highly desirable to have such an example, if any. 

It was anyway brought to our attention by Y. Yuan that one can at least construct a complete example: the holomorphic sectional curvature of the Bergman metric on the symmetrized bidisc is negatively pinched while the holomorphic bisectional curvature is positive somewhere \cite[Theorem 1]{CY20}.

\end{remark}

\section{Motivations from birational geometry}

We take the opportunity here to reproduce a quite standard argument in order to see how to reduce the Kobayashi conjecture on the ampleness of the canonical bundle of compact projective hyperbolic manifolds to showing that projective manifolds $X$ with trivial first real Chern class are not hyperbolic. 
The same strategy can also be applied to have a proof using techniques stemming from birational geometry of the Yau conjecture. All this, provided the abundance conjecture is true. Indeed, a slightly more general statement can be obtained and also the same kind of arguments can be applied to compact Kähler manifolds, as we shall see.

Let us begin by recalling what Abundance Conjecture predicts.

\begin{conjecture}[Abundance conjecture]
Let $X$ be a projective variety with at most Kawamata log-terminal singularities and with nef canonical bundle. Then, its canonical bundle is semi-ample, \textsl{i.e.} a large tensor power $K_X^{\otimes m}$ is generated by its global sections.
\end{conjecture}

This is known in dimension at most three. Observe that a smooth projective manifold has no singularities at all, so that the conclusion of this conjecture should in particular hold for such manifolds. 

\begin{remark}
In the special case of zero Kodaira dimension this implies the following. Suppose $X$ is a projective manifold, with nef canonical bundle $K_X$ and $\kappa(X)=0$. On one side, if abundance is true, we now that some power of $K_X$ must be globally generated. On the other side, $\kappa(X)=0$ means that, for every positive tensor power of $K_X$, the space of its global holomorphic sections is either $0$ or $1$, but not all can have dimension $0$. Therefore, the globally generated tensor powers have, up to multiples, only one section, and thus this section can never vanish. We deduce that such powers are holomorphically trivial. In particular $K_X$ is torsion, and $c_1(X)=-c_1(K_X)$ must be zero in rational or real cohomology. 
\end{remark}

So, let $X$ be a smooth Kobayashi hyperbolic projective manifold. By the celebrated criterion of Mori, $K_X$ is nef -- otherwise $X$ would contain a rational curve. Thus, $K_X$ already lies in the closure of the ample cone. Suppose now, and for the rest of the section, that the abundance conjecture holds true (in particular all we are saying hold in dimension at most three, unconditionally). Thus, we have that $K_X$ is semi-ample. We are therefore able to use the following for the canonical bundle.

\begin{theorem}[Semiample Iitaka fibrations {\cite[Theorem 2.1.27]{Laz04}}]
Let $X$ be a normal projective variety and let $L\to X$ be a semi-ample line bundle on $X$. Then, there is an algebraic fiber space (\textsl{i.e.} a projective surjective mapping with connected fibers)
$$
\phi\colon X\to Y,
$$
with $\dim Y=\kappa(L)$, having the property that, for all sufficiently big and divisible integers $m$, it coincides with the map associated to the complete linear system $|L^{\otimes m}|$. Furthermore, there is an integer $f$ and an ample line bundle $A$ on $Y$ such that $L^{\otimes f}\simeq\phi^*A$.
\end{theorem}

Now, since $K_X$ is semi-ample, we get an algebraic fiber space on $X$
$$
\phi\colon X\to Y,
$$
with $\dim Y=\kappa(X)$, and such that some power, say $K_X^{\otimes f}$, of the canonical bundle is the pull-back of an ample divisor $A$ on $Y$. In particular, for every general (hence smooth: recall that we are working with a smooth $X$) fiber $F$ of $\phi$, we have by taking the determinant of the short exact sequence
$$
0\to T_F\to T_X|_F\to\mathcal O_F^{\oplus\kappa(X)}\to 0,
$$
that $K_F\simeq K_X|_F$ (this is because obviously the normal bundle of the fiber is trivial). But then, $K_F^{\otimes f}\simeq K_X^{\otimes f}|_F\simeq\phi^*A|_F\simeq\mathcal O_F$. Thus, $K_F$ is torsion, and the general fiber has Kodaira dimension zero and trivial first Chern class in real cohomology.

\begin{conjecture}
A projective manifold with trivial real first Chern class is not Kobayashi hyperbolic.
\end{conjecture}

Suppose to be able to prove the above conjecture. We claim that this implies that the Kodaira dimension of a projective Kobayashi hyperbolic manifold $X$ must be maximal, that is, $X$ is of general type. Indeed, if $1\le\kappa(X)<\dim X$, then $\phi$ has positive dimensional fibers and the general ones have trivial real first Chern class, as observed above. So we would, by our assumptions, find a non-Kobayashi hyperbolic positive dimensional subvariety of $X$, contradiction.

Now, if $K_X$ is big and there are no rational curves on $X$ it is not difficult to show that $K_X$ is ample (cf. Lemma \ref{lem:ratcurv}), and we are done.

Next, how to prove that a projective manifold $X$ with trivial real first Chern class is not hyperbolic? By the Beauville--Bogomolov decomposition theorem \cite{Bea83}, a compact Kähler manifold with vanishing real first Chern class is, up to finite étale covers, a product of complex tori, Calabi--Yau manifolds and irreducible holomorphic symplectic manifolds. Since complex tori are obviously not Kobayashi hyperbolic, one is reduced to showing that Calabi--Yau manifolds and irreducible holomorphic symplectic manifolds are not Kobayashi hyperbolic (since Kobayashi hyperbolicity is preserved under étale covers). Very recently, in the spectacular paper \cite{Ver15}, Verbitsky has shown ---among other things--- that irreducible holomorphic symplectic manifolds with second Betti number greater than three (a condition that should indeed conjecturally hold for every irreducible holomorphic symplectic manifold) are not Kobayashi hyperbolic. 

Thus, one of the main challenges left is to show non hyperbolicity of Calabi--Yau manifolds. For such manifolds, much more is expected to be true: they should always contain rational curves! For several results in this direction, at least for Calabi--Yau manifolds with large Picard number, we refer the reader to \cite{Wil89,Pet91,HBW92,Ogu93,DF14,DFM16}, to cite only a few.

Coming back to curvature, of course possessing a Kähler metric whose holomorphic sectional curvature is negative implies Kobayashi hyperbolicity and thus having ample canonical bundle by the above discussion, provided the abundance conjecture is true. Therefore, this settles Yau's conjecture under the assumptions that abundance conjecture is true.

Now, what about compact Kähler manifolds with merely non positive holomorphic sectional curvature? Surely, they do not contain any rational curve (cf. Theorem \ref{algcrit}). Thus, if $X$ is projective, we conclude that $K_X$ is nef as before by Mori. If $X$ is merely Kähler, one needs to work more but the same conclusion of nefness for $K_X$ holds true, thanks to a very recent result by Tosatti and Yang \cite{TY15} (which is a slight modification of the original Wu and Yau method \cite{WY16}). Anyway, such a condition is not strong enough in order to obtain positivity of the canonical bundle, as flat complex tori immediately show. A less obvious but still easy counterexample is given by the product (with the product metric) of a flat torus and, say, a compact Riemann surface of genus greater than or equal to two endowed with its Poincaré metric. In this example, over each point there are some directions along which the holomorphic sectional curvature is strictly negative but always some flat directions, too (we refer the reader to the recent paper \cite{HLW14} for some nice results about this merely non positive case). 
So, if we look for the weakest condition, as long as the sign of holomorphic sectional curvature is concerned, for which one can hope to obtain the positivity of the canonical bundle, we are led to give the following (standard, indeed) definition.

\begin{definition}
The holomorphic sectional curvature is said to be \emph{quasi-negative} if $\HSC_\omega\le 0$ and moreover there exists at least one point $x\in X$ such that $\HSC_\omega(x,[v])< 0$ for every $v\in T_{X,x}\setminus\{0\}$.
\end{definition}

\begin{remark}
A slightly subtler condition on holomorphic sectional curvature that should also work for this kind of purposes, based on the notion of \lq\lq truly flat\rq\rq{} directions, was very recently introduced by Heier, Lu, Wong, and Zheng in \cite{HLWZ17}. We refer directly to that paper for more on this.
\end{remark}

Now, why should we hope that such a condition would be sufficient? The reason comes again from the birational geometry of complex Kähler manifolds, and in particular again from the abundance conjecture. Let us illustrate why.

We begin with the following elementary observation.

\begin{proposition}\label{prop:average}
Let $(X,\omega)$ be a compact Kähler manifold with $\HSC_\omega\le 0$, and suppose there exists a direction $[v]\in P(T_{X,x_0})$ such that $\HSC_\omega(x_0,[v])<0$, for some $x_0\in X$. Then, $c_1(X)\in H^2(X,\mathbb R)$ cannot be zero.
\end{proposition}

\begin{proof}
By Proposition \ref{average}, we know that $\scal_\omega$ is everywhere non positive, and moreover, as an average, it is strictly negative at $x_0$. In particular, the total scalar curvature of $\omega$ is strictly negative. The conclusion follows from Remark \ref{totscalcurv}.
\end{proof}

As a direct consequence, if $X$ is moreover projective and $\operatorname{Pic}(X)$ is infinite cyclic, then $K_X$ must be ample. This way one can recover (and slightly generalize) a result of \cite{WWY12}.

Now, let $(X,\omega)$ be a compact Kähler manifold with quasi-negative holomorphic sectional curvature. Then, Proposition \ref{prop:average} implies that $X$ cannot have trivial first real Chern class. Moreover, since being quasi-negative is stronger than being non positive, we saw that, thanks to \cite{TY15}, $K_X$ is nef.

Once again, suppose that the abundance conjecture holds true, but now also for compact Kähler manifolds. Then, $K_X$ is semi-ample and we can consider exactly as before the semi-ample Iitaka fibration for $K_X$.
Since $X$ has non trivial first real Chern class, we must have that $\kappa(X)>0$, otherwise some power of the canonical bundle would be a pull-back of a (ample) line bundle over point, and thus would be trivial!

If $\kappa(X)=\dim X$, then $X$ would be birational to a projective variety, \textsl{i.e.} would be a Moishezon manifold. By Moishezon's theorem, a compact Kähler Moishezon manifold is projective. Moreover, $X$ contains no rational curves and is of general type, and we conclude as before that $K_X$ must be ample.

Next, suppose by contradiction that $1\le\kappa(X)\le \dim X-1$ so that if we call $F$ the general fiber of $\phi$, we have that $F$ is a smooth compact Kähler manifold of positive dimension and different from $X$ itself. 
Now, on the one hand, the short exact sequence of the fibration shows that
$K_F\simeq K_X|_F$ and therefore it follows that $c_1(F)$ must be zero in real cohomology. On the other hand, the classical Griffiths' formulae for curvature of holomorphic vector bundles imply that the holomorphic sectional curvature decreases when passing to submanifolds, that is for every $x\in F\subset X$
$$
\HSC_{\omega|_F}(x,[v])\le\HSC_{\omega}(x,[v]),
$$
where $v\in T_{F,x}$ and, in the right hand side, $v$ is seen as a tangent vector to $X$.

The quasi-negativity of the holomorphic sectional curvature implies, since $F$ is a general fiber, that there exists a tangent vector to $F$ along which the holomorphic sectional curvature of $\omega|_F$ is strictly negative. Thus, Proposition \ref{prop:average} implies that $F$ cannot have trivial first real Chern class, which is absurd.

As a consequence, me may indeed hope to extend Wu--Yau--Tosatti--Yang theorem to the optimal, quasi-negative case. This is precisely the main contribution of the paper \cite{DT16}.

\begin{theorem}[{\cite[Theorem 1.2]{DT16}}]\label{thm:main}
Let $(X,\omega)$ be a connected compact Kähler manifold. Suppose that the holomorphic sectional curvature of $\omega$ is quasi-negative. Then, $K_X$ is ample. In particular, $X$ is projective.
\end{theorem}

We shall spend some words on this result as well as the Kähler case in the last section. 

\begin{remark}
To finish this section with, unfortunately, we must confess that we are not aware of any example of a compact Kähler manifold with a Kähler metric whose holomorphic sectional curvature is quasi-negative but which does not posses any Kähler metric with strictly negative holomorphic sectional curvature. In other word, is Theorem \ref{thm:main} a true generalization of Wu--Yau--Tosatti--Yang result? We believe so. Then, such an example, if any, would be urgently needed!
\end{remark}
\section{An example by J.-P. Demailly}

In this section we would like to explain, following \cite[\S8]{Dem97}, how one can construct examples of compact hyperbolic projective manifolds which nevertheless do not admit any hermitian metric of negative holomorphic sectional curvature. Such examples can be generalized to higher order analogues ---namely \lq\lq$k$-jet curvature\rq\rq{}--- of holomorphic sectional curvature: this will be mentioned at the end of the section, and related conjectures that come out from this picture will be discussed at the end of the chapter.

The first observation is the following algebraic criterium for the nonexistence of a metric with negative holomorphic sectional curvature. Let $X$ be a complex manifold, $C$ be a compact Riemann surface, and $F\colon C\to X$ be a non-constant holomorphic map. Let $m_p\in\mathbb N$ be the multiplicity at $p\in C$ of $F$. Clearly, $m_p=1$ except possibly at finitely many points of $C$, and $m_p\ge 2$ if and only if $F$ is not an immersion at $p$. 

\begin{theorem}[Demailly {\cite[Special case of Theorem 8.1]{Dem97}}]\label{algcrit}
Consider \newline$(X,\omega)$ a compact hermitian manifold and let $F\colon C\to X$ be a non constant holomorphic map from a compact Riemann surface $C$ of genus $g=g(C)$ to $X$. Suppose that $\operatorname{HSC}_\omega\le-\kappa$ for some $\kappa\ge 0$. Then,
$$
2g-2\ge\frac{\kappa}{2\pi}\deg_\omega C+\sum_{p\in C}(m_p-1),
$$ 
where $\deg_\omega C=\int_C F^*\omega>0$ is the degree of $C$ with respect to $\omega$.
\end{theorem}

We shall use this theorem especially in the case where $C$ is the normalization of a singular curve in $X$ and $F$ the normalization map. Observe that, in particular, we recover the well-known fact that on a compact hermitian manifold with negative holomorphic sectional curvature there are no rational nor elliptic curves (even singular), and that there are no rational (possibly singular) curves on a compact hermitian manifold with non positive holomorphic sectional curvature.

The main ingredients of the proof are two: first, as already observed, curvature decreases one passing to subbundles (even if one needs some adjustments here, since at the very beginning one merely gets a locally free subsheaf and not an actual subbundle), and second, the holomorphic sectional curvature for a Riemann surface is nothing else than the gaussian curvature so that in the compact case its total integral gives the opposite of its canonical degree, by the Gauss--Bonnet Theorem and the Hurwitz Formula.

\begin{proof}
The differential $F'$ of $F$ gives us a map $F'\colon T_C\to F^* T_X$. This map is injective at the level of sheaves, but not necessarily at the level of vector bundle, since $F'$ may vanish at some point. Taking into account these vanishing points counted with multiplicities, we obtain the following injection of vector bundles
$$
F'\colon T_C\otimes\mathcal O_C(D)\to F^*T_X,
$$
where we defined the effective divisor $D$ to be $\sum_{p\in C}(m_p-1)\, p$. 
Thus, via $F'$, we realized  $T_C\otimes\mathcal O_C(D)$ as a subbundle of the hermitian vector bundle $(F^*T_X, F^*\omega)$, and ---as such--- we can endow it with the induced metric $h=F^*\omega|_{T_C\otimes\mathcal O_C(D)}$ (observe the we consider $F^*\omega$ not as a pull-back of differential forms, but as a pull-back of hermitian metrics). 

Now, a local holomorphic frame for $T_C\otimes\mathcal O_C(D)$ around a point $p\in C$ is given by $\eta(t)=1/t^{m_p-1}\,\frac{\partial}{\partial t}$, where $t$ is a holomorphic coordinate centered at $p$. Call $\xi(t)=F'(\eta(t))\in (F^*T_X)_t=T_{X,F(t)}$, so that $\xi$ is a local holomorphic frame for $T_C\otimes\mathcal O_C(D)$ when seen as a subbundle of $F^*T_X$. We have, for the Griffiths curvature of $(T_C\otimes\mathcal O_C(D),h)$,
\begin{multline*}
\langle\Theta(T_C\otimes\mathcal O_C(D),h)(\partial/\partial t,\partial/\partial \bar t\,)\cdot \xi,\xi\rangle_h \\
=\Theta(T_C\otimes\mathcal O_C(D),h)(\partial/\partial t,\partial/\partial \bar t\,)\,\underbrace{||\xi||^2_h}_{=||\xi||^2_\omega}.
\end{multline*}
By the classical Griffiths' formulae, we have the following decreasing property for the Griffiths curvatures:
\begin{multline*}
\langle\Theta(T_C\otimes\mathcal O_C(D),h)(\partial/\partial t,\partial/\partial \bar t\,)\cdot \xi,\xi\rangle_h \\
\le \langle\Theta(F^*T_X,F^*\omega)(\partial/\partial t,\partial/\partial \bar t\,)\cdot \xi,\xi\rangle_{F^*\omega} \\
=  \langle F^*\Theta(T_X,\omega)(\partial/\partial t,\partial/\partial \bar t\,)\cdot \xi,\xi\rangle_{F^*\omega}\\
= \langle\Theta(T_X,\omega)(F'(\partial/\partial t),\overline{F'(\partial/\partial t)}\,)\cdot \xi,\xi\rangle_{\omega}\\
=|t^{m_p-1}|^2\,\langle\Theta(T_X,\omega)(\xi,\bar\xi\,)\cdot \xi,\xi\rangle_{\omega}
\le -\kappa |t^{m_p-1}|^2 ||\xi||^4_\omega,
\end{multline*}
where the last inequality holds since $\langle\Theta(T_X,\omega)(\xi,\bar\xi\,)\cdot \xi,\xi\rangle_{\omega}=||\xi||^4_\omega\,\operatorname{HSC}_\omega(\xi)$. Therefore, we obtain
$$
\Theta(T_C\otimes\mathcal O_C(D),h)(\partial/\partial t,\partial/\partial \bar t\,)\le-\kappa |t^{m_p-1}|^2 ||\xi||^2_\omega=i\kappa\,(F^*\omega)(\partial/\partial t,\partial/\partial \bar t\,),
$$
where by $F^*\omega$ here we mean the pull-back at the level of differential forms. Summing up, we have obtained that
$$
i\,\Theta(T_C\otimes\mathcal O_C(D),h)\le -\kappa\,F^*\omega,
$$
as real $(1,1)$-forms.
But then, 
$$
\int_C\frac{i}{2\pi}\,\Theta(T_C\otimes\mathcal O_C(D),h)\le-\frac{\kappa}{2\pi}\int_C F^*\omega=-\frac{\kappa}{2\pi}\deg_\omega C,
$$
and 
$$
\int_C\frac{i}{2\pi}\,\Theta(T_C\otimes\mathcal O_C(D),h)=\deg(T_C\otimes\mathcal O_C(D))=2-2g+\sum_{p\in C}(m_p-1),
$$
since $\deg(T_C)=2-2g$ by Hurwitz's formula. The statement follows.
\end{proof}

Following Demailly, we shall now exhibit a smooth projective surface which is Kobayashi hyperbolic, has an ample canonical bundle, but which cannot admit any hermitian metric with negative holomorphic sectional curvature. It will be constructed as a fibration of Kobayashi hyperbolic curves onto a Kobayashi hyperbolic curve, with at least one \lq\lq very\rq\rq{} singular fiber, which will violate the above criterion.

\begin{proposition}[Cf. {\cite[8.2. Theorem]{Dem97}}]
There exists a smooth projective surface $S$ which is hyperbolic (and hence with ample canonical bundle $K_S$) but does not carry any hermitian metric with negative holomorphic sectional curvature. Moreover, given any two smooth compact hyperbolic Riemann surfaces $\Gamma, \Gamma'$, such a surface can be obtained as a fibration $S\to\Gamma$, with hyperbolic fibers, in which (at least) one of the fibers is singular and has $\Gamma'$ as its normalization. 
\end{proposition}

\begin{proof}
Take any compact hyperbolic Riemann surface $\Gamma'$, and let $g=g(\Gamma')\ge 2$ be its genus. Now, we modify it into a singular compact Riemann surface $\Gamma''$ of the same genus, whose normalization is $\Gamma'$. 

In order to do so, consider a pair of positive relatively prime integers $(a,b)$, with $a<b$, and the associated affine plane curve $C$ in $\mathbb C^2$ given by the equation $y^a-x^b=0$, which has a monomial singularity of type $(a,b)$ at $0\in\mathbb C^2$. Its normalization is given by $\mathbb C\ni t\mapsto (t^a,t^b)\in\mathbb C^2$. Choose integers $n,m$ such that $na+mb=1$. Then, the restriction of the rational function on $\mathbb C^2$ defined by $(x,y)\mapsto x^ny^m$ to $C$ gives a holomorphic coordinate on it minus the singular point (this is actually the inverse map of the normalization map outside the singularity). In particular, the set of points $(x,y)\in C$ such that $0<|x|<1$ is biholomorphic to the punctured unit disc.

Now, take a point $x_0\in\Gamma'$ and choose a holomorphic coordinate centered at $x_0$ such that we can select a neighborhood of $x_0$ whose image is the unit disc \textsl{via} this coordinate. Finally, remove the point $x_0$ in order to obtain a holomorphic coordinate chart whose image is the punctured unit disc. By identifying with the punctured unit disc constructed above, we replace this neighborhood of $x_0$ with the set of point $(x,y)\in C$ such that $|x|<1$, thus creating the desired singularity at $x_0$. Call the resulting curve $\Gamma''$. By construction, the normalization of $\Gamma''$ is exactly $\Gamma'$, and $\Gamma''$ has a single singular point, whose singularity type is plane and monomial of type $(a,b)$ (for an excellent and very elementary discussion around this subject we refer the reader to \cite[Chapter III, Section 2]{Mir95}).

Next, we embed $\Gamma''$ in some large projective space, and then we project it to $\mathbb P^2$, in such a way that the singular point is left untouched and outside it we create at most a finite number of nodes (\textsl{i.e.} plane monomial singularity of type $(2,2)$). Call the resulting projective plane curve $C_0$, whose normalization is of course again $\Gamma'$. Observe that the normalization map $\nu\colon\Gamma'\to C_0$ is an immersion outside the (single) preimage of the first singular point we created. On the other hand, at this point it has multiplicity $a$.

In order to obtain the desired surface $S$, we select then $a$ so that $a-1>2g-2$, \textsl{i.e.} $a\ge 2g$. Such a surface $S$ then does contain a curve which violates the criterium given in Theorem \ref{algcrit}, and we are done.

Take a (reduced) homogeneous polynomial equation $P_0(z_0,z_1,z_2)=0$ for $C_0$ in $\mathbb P^2$. Then, we necessarily have $d=\deg P_0\ge 4$, since otherwise $C_0$ would be normalized by a rational or an elliptic curve. Next, complete $P_0$ into a basis $\{P_0,P_1,\dots,P_N\}$ of the space $H^0(\mathbb P^2,\mathcal O(d))$ of homogeneous polynomials of degree $d$ in three variables, and consider the corresponding universal family 
$$
\mathcal U=\bigl\{\bigl([z_0:z_1:z_2],[\alpha_0:\cdots:\alpha_N]\bigr)\in\mathbb P^2\times\mathbb P^N\mid\sum_{j=0}^N\alpha_j\,P_j(z)=0\bigr\}\subset\mathbb P^2\times\mathbb P^N,
$$
of curves of degree $d$ in $\mathbb P^2$, together with the projection $\pi\colon\mathcal U\to\mathbb P^N$. Our starting curve $C_0$ is then the fiber $U_{[1:0:\cdots:0]}$ over the point $[1:0:\cdots:0]\in\mathbb P^N$. Now, we embed the first curve $\Gamma$ into $\mathbb P^N$ (this is of course possible since $N\ge 3$) in such a way that $[1:0:\cdots:0]\in\Gamma$. The desired fibration $S\to\Gamma$ will be obtained as the pull-back family
$$
\xymatrix{S=\mathcal U\times_{\mathbb P^N}\Gamma \ar@{->}[r]\ar@{->}[d]& \mathcal U \ar@{->}[d] \\ \Gamma \ar@{^{(}->}[r]& \mathbb P^N.}
$$
Of course, we have to select carefully the embedding of $\Gamma$ into $\mathbb P^N$, so that $S$ will be non singular, and in such a way that we have a good control of the singular fibers out of $U_{[1:0:\cdots:0]}$.

In order to do so, the first observation is that ---as it is well-known--- the locus $Z$ in $\mathbb P^N$ which corresponds to singular curve is an algebraic hypersurface and, moreover, the locus $Z'\subset Z$ which corresponds to curves which have not only one node in their singularity set is of codimension $2$ in $\mathbb P^N$. In particular, by possibly moving $\Gamma$ with a generic projective automorphism of $\mathbb P^N$ leaving fixed $[1:0:\cdots:0]$, we can suppose that $\Gamma\cap Z'=\{[1:0:\cdots:0]\}$, so that all the fibers of $S$, except from $C_0$, are either smooth, or with a single node. If such an $S$ were non singular, we would be done. Indeed, by Plücker's formula, the smooth fibers have genus $(d-1)(d-2)/2\ge 3$, $U_{[1:0:\cdots:0]}$ has genus $g\ge 2$ by construction, and the other singular fibers have genus $(d-1)(d-2)/2-1\ge 2$, since they have only one node. Therefore, $S$ is a fibration onto a hyperbolic Riemann surface with all hyperbolic fibers and is then hyperbolic (and hence with ample canonical bundle), with a fiber which contradicts Theorem \ref{algcrit}.

So we are left to checking the smoothness of $S$, knowing that we can possibly use again generic automorphisms of $\mathbb P^N$ leaving fixed $[1:0:\cdots:0]$ to move $\Gamma$. Thus, since $\Gamma$ is embedded in $\mathbb P^N$, we can think of $S$ as included in $\mathcal U$, and since $\mathcal U$ is smooth, Bertini's theorem immediately implies that $S$ can be chosen non singular outside $U_{[1:0:\cdots:0]}$. Now, what about points along $U_{[1:0:\cdots:0]}$? Fix such a point $([z_0:z_1:z_2],[1:0:\cdots:0])\in U_{[1:0:\cdots:0]}$, and suppose, just to fix ideas, that $z_0\ne 0$. Take the corresponding affine coordinates, say $((z,w),(a_1,\dots,a_N))$ around this point, set $p_j(z,w)=P_j(1,z,w)$ to be the dehomogenization of the $P_j$'s, and let $f_1(a),\dots,f_r(a)$ be affine equations of the curve $\Gamma$. Then, we have to check the rank of the following Jacobian matrix at the point $\bigl((z,w),(0,\dots,0)\bigr)$, the affine equation for $\mathcal U$ being $p_0(z,w)+\sum_{j=1}^N a_j\,p_j(z,w)=0$:
$$
\begin{pmatrix}
\frac{\partial p_0}{\partial z} (z,w)& \frac{\partial p_0}{\partial w} (z,w) & p_1(z,w) & \cdots & p_N(z,w) \\ 
0 & 0 & \frac{\partial f_1}{\partial a_1}(0) & \cdots & \frac{\partial f_1}{\partial a_N}(0) \\
\vdots & \vdots & \vdots &  & \vdots \\
0 & 0 & \frac{\partial f_r}{\partial a_1}(0) & \cdots & \frac{\partial f_r}{\partial a_N}(0)
\end{pmatrix}.
$$
Observe that the lower right block has rank equal to $N-1$ being $\Gamma$ smooth. Call this block $A$ and let $v=(v_1,\dots,v_n)\in\mathbb C^N$ be a generator for the kernel of this block, which is thus a nonzero tangent vector to $\Gamma$ at $0$. In order to get rank $N+2-2=N$ for the entire Jacobian matrix we have only to worry about (the finitely many) singular points of $C_0=U_{[1:0:\cdots:0]}$, since at regular points either $\frac{\partial p_0}{\partial z}$ or $\frac{\partial p_0}{\partial w}$ is non zero. If $(z,w)$ is a singular point for $U_{[1:0:\cdots:0]}$, then the condition for $S$ to be smooth around this point is given by $(p_1,\dots,p_N)^t\not\in \operatorname{Im}(A^t)=\bigl(\overline{\ker A}\bigr)^\perp$, that is
$$
\sum_{j=1}^N v_j\,p_j(z,w)\ne 0.
$$
This can be of course achieved by possibly moving again $\Gamma$ with a generic projective automorphism of $\mathbb P^N$ leaving fixed $[1:0:\cdots:0]$, since only the tangent line of $\Gamma$ at $0$ is concerned in the required condition.
\end{proof}
\section{The Wu--Yau theorem}

In this section we go into the details of the proof of Wu--Yau's theorem on the positivity of the canonical class for projective manifolds endowed with a Kähler metric of negative holomorphic sectional curvature. We will present a proof which follows, for the first part, almost \textsl{verbatim} the original proof of Wu and Yau. On the other hand, the conclusion will be achieved with an approach which is more pluripotential in flavor, taken from \cite{DT16}. Finally, we shall discuss at the end of this section several generalizations of this result (including the Kähler case, and weaker notions of negativity).

The proof is achieved in essentially three steps, after a reduction as follows. As we have seen, the negativity of the curvature (or even its non-positivity) implies the non existence of rational curves on $X$. Then, by Mori's Cone Theorem, we deduce that the canonical bundle of $X$ is nef. But then, it is sufficient to prove that $c_1(K_X)^n>0$, which in this case means that the canonical bundle is big. Indeed, if $K_X$ is big and there are no rational curves on $X$ one can conclude the ampleness of the canonical bundle via the following standard lemma.

\begin{lemma}[Exercise 8, page 219 of \cite{Deb01}]\label{lem:ratcurv}
Let $X$ be a smooth projective variety of general type which contains no rational curves. Then, $K_X$ is ample.
\end{lemma}

Here is a proof, for the sake of completeness.

\begin{proof}
Since there are no rational curves on $X$, Mori's theorem implies as above that $K_X$ is nef. Since $K_X$ is big and nef, the Base Point Free theorem tells us that $K_X$ is semi-ample. If $K_X$ were not ample, then the morphism defined by (some multiple of) $K_X$ would be birational but not an isomorphism. In particular, there would exist an irreducible curve $C\subset X$ contracted by this morphism. Therefore, $K_X\cdot C=0$. Now, take any ample divisor $H$. For any $\varepsilon>0$ rational and small enough, $K_X-\varepsilon H$ remains big and thus some large positive multiple, say $m(K_X-\varepsilon H)$, of $K_X-\varepsilon H$ is linearly equivalent to an effective divisor $D$. Set $\Delta=\varepsilon' D$, where $\varepsilon'>0$ is a rational number.  We have:
$$
\begin{aligned}
(K_X+\Delta)\cdot C & =\varepsilon'\,D\cdot C \\
& =\varepsilon'm(K_X-\varepsilon H)\cdot C \\
& =-\varepsilon\varepsilon'm\,H\cdot C<0.
\end{aligned}
$$
Finally, if $\varepsilon'$ is small enough, then $(X,\Delta)$ is a klt pair. Thus, the (logarithmic version of the) Cone Theorem would give the existence of an extremal ray generated by the class of a rational curve in $X$, contradiction.
\end{proof}

\subsection{Description of the main steps of the proof.}

Keeping in mind that what we have to show is that $c_1(K_X)^n>0$, we illustrate now the steps of the proof.

\subsubsection*{Step 1: Solving an approximate Kähler--Einstein equation.}

Let $\omega$ be our fixed Kähler metric (with negative holomorphic sectional curvature, but we shall not use this hypothesis for the moment).

\begin{claim}\label{step1}
For each $\varepsilon>0$ there exists a unique smooth function $u_\varepsilon \colon X\to\mathbb R$ such that
$$
\omega_\varepsilon:=\varepsilon\omega-\operatorname{Ric}_\omega+\frac{i}{2\pi}\partial\bar\partial u_{\varepsilon}
$$
is a positive $(1,1)$-form (hence Kähler, belonging to the cohomology class $c_1(K_X)+\varepsilon[\omega])$ form satisfying the Monge--Ampère equation
$$
\omega_\varepsilon^n=e^{u_\varepsilon}\,\omega^n.
$$
In particular, 
$$
\operatorname{Ric}(\omega_\varepsilon)=-\omega_\varepsilon+\varepsilon\omega,
$$
whence the terminology \lq\lq approximate Kähler--Einstein\rq\rq{}, and we have the following uniform upper bound:
$$
\sup_X u_\varepsilon\le C,
$$
where the constant $C$ depends only on $\omega$ and $n=\dim X$. Observe finally, that in particular, $\operatorname{Ric}(\omega_\varepsilon)\ge-\omega_\varepsilon$.
\end{claim}

\subsubsection*{Step 2: A laplacian estimate involving the holomorphic sectional curvature.} This step is somehow a refinement of the laplacian estimate needed in order to achieve the classical $C^2$-estimates to solve the complex Monge--Ampère equation on compact Kähler manifolds. In the classical setting an upper bound for the holomorphic bisectional curvature is used. Here we shall employ a lemma due to Royden in order to use only the weaker information given by the bound on the holomorphic sectional curvature, as in the hypotheses. The crucial part of this step is the following.

\begin{claim}\label{step2}
Suppose $-\kappa<0$ is an upper bound for the holomorphic sectional curvature of $\omega$. Suppose moreover that $\omega'$ is another Kähler metric on $X$ whose Ricci curvature is comparable with $\omega$ and $\omega'$ as follows:
$$
\operatorname{Ric}(\omega')\ge-\lambda\omega'+\mu\omega,
$$
where $\lambda,\mu$ are non negative constants. Define a smooth function $S\colon X\to\mathbb R_{>0}$ to be the trace of $\omega$ with respect to $\omega'$, \textsl{i.e.}
$$
S:=\operatorname{tr}_{\omega'}\omega=n\,\frac{\bigl(\omega'\bigr)^{n-1}\wedge\omega}{\bigl(\omega'\bigr)^{n}}.
$$
Then, the following differential inequality holds:
\begin{equation}\label{diffineq}
-\Delta_{\omega'}\log S\ge\biggl(\frac{\kappa(n+1)}{2n}+2\pi\frac\mu n\biggr)S-2\pi\lambda.
\end{equation}
\end{claim}

Observing that
$$
\int_X\Delta_{\omega'}\log S\,\bigl(\omega'\bigr)^n=0,
$$
we shall use the inequality above with $\omega'=\omega_\varepsilon$, $\lambda=1$, and $\mu=0$, in the following integral form:
\begin{equation}\label{intineq}
\int_X\frac{(n+1)\kappa}{2n}S_\varepsilon\,\omega_\varepsilon^n\le 2\pi\int_X\omega_\varepsilon^n,
\end{equation}
where we added the subscript $\varepsilon$ to $S$ in order to emphasize the dependence of $S$ from $\varepsilon$.

\subsubsection*{Step 3: Proof of the key inequality.}

One wants to show that $c_1(K_X)^n>0$. Since $\omega_\varepsilon=-\operatorname{Ric}(\omega_\varepsilon)+\varepsilon\omega$, we have that
$$
\omega_\varepsilon^n=\bigl(-\operatorname{Ric}(\omega_\varepsilon)\bigr)^n+\varepsilon\sum_{j=1}^n\binom{n}{j}\varepsilon^{j-1}\,\bigl(-\operatorname{Ric}(\omega_\varepsilon)\bigr)^{n-j}\wedge\omega^j.
$$
But then,
$$
\int_X\omega_\varepsilon^n=\int_X\bigl(-\operatorname{Ric}(\omega_\varepsilon)\bigr)^n+\varepsilon\sum_{j=1}^n\binom{n}{j}\varepsilon^{j-1}\int_X \bigl(-\operatorname{Ric}(\omega_\varepsilon)\bigr)^{n-j}\wedge\omega^j.
$$
On the other hand, since the cohomology class $\bigl[-\operatorname{Ric}(\omega_\varepsilon)\bigr]=c_1(K_X)$ is independent from $\varepsilon$, the integrals
$$
\int_X \bigl(-\operatorname{Ric}(\omega_\varepsilon)\bigr)^{n-j}\wedge\omega^j=c_1(K_X)^{n-j}\cdot [\omega]^j
$$
are purely cohomological, so that
$$
\int_X\omega_\varepsilon^n=\underbrace{\int_X\bigl(-\operatorname{Ric}(\omega_\varepsilon)\bigr)^n}_{=c_1(K_X)^n}+O(\varepsilon),
$$
and thus
$$
c_1(K_X)^n=\lim_{\varepsilon\to 0^+}\int_X\omega_\varepsilon^n.
$$
What we want is therefore to show the positivity of such a limit.

\begin{claim}\label{step3}
The limit
$$
\lim_{\varepsilon\to 0^+}\int_X\omega_\varepsilon^n
$$
is strictly positive.
\end{claim}
This is what we call the key inequality. During the proof of the main result, this will be the only step where what we present here differs from Wu--Yau's original approach.

\subsection{Proof of the steps.} 
We now proceed with the proof of the various claims stated above.

\begin{proof}[Proof of Claim \ref{step1}]

The first observation is that, since $K_X$ is nef, for each $\varepsilon>0$ the cohomology class $c_1(K_X)+\varepsilon\,[\omega]=-c_1(X)+\varepsilon\,[\omega]$ is a Kähler class. This implies, thanks to the $\partial\bar\partial$-lemma, that there exists a smooth real function $f_\varepsilon$ on $X$, unique up to an additive constant, such that
$$
\omega_{f_\varepsilon}:=\varepsilon\,\omega-\Ric_\omega+\frac i{2\pi}\,\partial\bar\partial f_\varepsilon
$$
is a Kähler form on $X$.

Now we use the following theorem, in order to obtain an approximate Kähler--Einstein metric on $X$. We give here Yau's original general statement, which is the key ingredient to get his celebrated solution of the Calabi conjecture.

\begin{theorem}[Yau \cite{Yau78}]\label{solCalabi}
Let $(X,\omega_0)$ be a compact Kähler manifold, and $F\colon X\times\mathbb R_t\to\mathbb R$ smooth function such that $\partial F/\partial t\ge 0$. Suppose that there exists smooth function $\psi\colon X\to\mathbb R$ such that
$$
\int_X e^{F(x,\psi(x))}\,\omega_0^n=\int_X\omega^n_0.
$$
Then, there exists a unique (up to a constant if $F$ does not actually depend on $\psi$) smooth function $\varphi\colon X\to\mathbb R$, such that
$$
\begin{cases}
\omega_0+\frac i{2\pi}\,\partial\bar\partial\varphi>0,\\
\bigl(\omega_0+\frac i{2\pi}\,\partial\bar\partial\varphi\bigr)^n=e^{F(x,\varphi(x))}\,\omega_0^n.
\end{cases}
$$
\end{theorem}

From this statement one can derive easily both the existence of a Kähler metric in a fixed Kähler class with prescribed volume form (or, equivalently, Ricci tensor), and the existence of Kähler--Einstein metrics on compact Kähler manifold with negative (resp. zero) real first Chern class.

Now, we fix $\varepsilon>0$, and define a smooth real function $\alpha_\varepsilon$ on $X$ implicitly by
$$
\omega_{f_\varepsilon}^n=e^{-\alpha_\varepsilon}\,\omega^n.
$$
We then apply the theorem above with the following data:
$$
\omega=\omega_{f_\varepsilon},\quad F(x,t)=t+\alpha_\varepsilon(x)+f_\varepsilon(x).
$$
Then, there exists a unique smooth real function $v_\varepsilon$ such that
$$
\begin{aligned}
\biggl(\omega_{f_\varepsilon}+\frac i{2\pi}\,\partial\bar\partial v_\varepsilon\biggr)^n &= e^{v_\varepsilon+\alpha_\varepsilon+f_\varepsilon}\,\omega_{f_\varepsilon}^n \\
&= e^{v_\varepsilon+f_\varepsilon}\,\omega^n,
\end{aligned}
$$
and
$$
\omega_\varepsilon:=\omega_{f_\varepsilon}+\frac i{2\pi}\,\partial\bar\partial v_\varepsilon>0
$$
on $X$. Now, define $u_\varepsilon$ to be the sum $f_\varepsilon+v_\varepsilon$, so that it holds
$$
\omega_\varepsilon^n=e^{u_\varepsilon}\,\omega^n.
$$
Thus, we get for the Ricci curvature of $\omega_\varepsilon$
$$
\begin{aligned}
\Ric_{\omega_\varepsilon} &= -\frac i{2\pi}\partial\bar\partial\log\omega_\varepsilon^n \\
&=-\frac i{2\pi}\partial\bar\partial v_\varepsilon\underbrace{-\frac i{2\pi}\partial\bar\partial f_\varepsilon\overbrace{-\frac i{2\pi}\partial\bar\partial\log\omega^n}^{=\Ric_\omega}}_{=\varepsilon\,\omega-\omega_{f_\varepsilon}}\\
&=\varepsilon\,\omega-\omega_\varepsilon.
\end{aligned}
$$
In particular, $\Ric_{\omega_\varepsilon}\ge-\omega_\varepsilon$.

Now, we use the maximum principle in order to obtain the desired uniform upper bound for $u_\varepsilon$. To do so, pick a point $x_0\in X$ such that $\sup_X u_\varepsilon=u_\varepsilon(x_0)$. Then, at this point we have that the complex hessian of $u_\varepsilon$ is negative semi-definite, \textsl{i.e.} $i\,\partial\bar\partial u_\varepsilon(x_0)\le 0$. Thus,
$$
\begin{aligned}
\omega_\varepsilon(x_0) &= \bigl(\varepsilon\,\omega-\Ric_\omega+\frac i{2\pi}\,\partial\bar\partial u_\varepsilon\bigr)(x_0)\\
&\le \bigl(\varepsilon\,\omega-\Ric_\omega\bigr)(x_0) \\
&\le \bigl(\varepsilon_0\,\omega-\Ric_\omega\bigr)(x_0),
\end{aligned}
$$
if $\varepsilon_0>\varepsilon$. Therefore,
$$
\begin{aligned}
e^{\sup_X u_\varepsilon} &= e^{u_\varepsilon(x_0)} \\
&= \frac{\bigl(\varepsilon\,\omega-\Ric_\omega+\frac i{2\pi}\,\partial\bar\partial u_\varepsilon\bigr)^n(x_0)}{\omega^n(x_0)} \\
&\le\frac{\bigl(\varepsilon_0\,\omega-\Ric_\omega\bigr)^n(x_0)}{\omega^n(x_0)}=:e^C,
\end{aligned}
$$
so that
$$
\sup_X u_\varepsilon\le C,\quad\forall\varepsilon<\varepsilon_0.
$$
This complete the proof of Claim \ref{step1}.
\end{proof}

\begin{proof}[Proof of Claim \ref{step2}]

Let $x_0\in X$ be a fixed point. Chose holomorphic normal coordinates $(z_1,\dots,z_n)$ with respect to $\omega$, centered at $x_0$. Without loss of generality, by a constant $\omega$-unitary change of variables, we may also suppose that $\omega'$ is diagonalized with respect to $\omega$ at $x_0$. Thus we write
$$
\omega=i\sum_{l,m=1}^n\omega_{lm}\,dz_l\wedge d\bar z_m, \quad\omega_{lm}(z)=\delta_{lm}-\sum_{j,k=1}^n c_{jklm}\,z_j\bar z_k +O(|z^3|),
$$
where the $c_{jklm}$'s are the coefficients of the Chern curvature tensor of $(X,\omega)$ at $x_0$, and
$$
\omega'=i\sum_{l,m=1}^n\omega'_{lm}\,dz_l\wedge d\bar z_m,\quad \omega'_{lm}(z)=\lambda_l\,\delta_{lm}+O(|z|),
$$
where the $\lambda_j$'s are the eigenvalues at $x_0$ of $\omega'$ with respect to $\omega$. In particular, $\lambda_j>0$, $j=1,\dots,n$. Next, call $\rho'_{jk}$ the coefficients of the Ricci curvature of $\omega'$, so that
$$
\Ric_{\omega'}=\frac i{2\pi}\sum_{j,k=1}^n \rho'_{jk}\,dz_j\wedge d\bar z_k,
$$
where
$$
\rho'_{jk}=\sum_{l=1}^nc'_{jkll}.
$$

With these notations, the starting point is the following

\begin{lemma}[See {\cite[pag. 371]{WYZ09}}]
The following differential equality holds:
\begin{equation}\label{difeq}
-\Delta_{\omega'}S(x_0)=\sum_{l=1}^n\frac{\rho'_{ll}}{(\lambda_l)^2}
+\sum_{j,l,a=1}^n\frac {\bigl|\partial\omega'_{al}/\partial z_j\bigr|^2}{\lambda_j(\lambda_l)^2\lambda_a}-\sum_{j,l=1}^n\frac{c_{jjll}}{\lambda_j\lambda_l},
\end{equation} 
where the right hand side is intended to be computed at $x_0$.
\end{lemma}

\begin{proof}
A straightforward computation, using the adjugate matrix method to obtain the inverse, shows that
$$
S=\sum_{l,m=1}^n\Omega'_{ml}\,\omega_{lm},
$$
where we define $(\Omega'_{lm})$ to be the inverse matrix of $(\omega'_{lm})$.
We want to compute $\Delta_{\omega'}S$ at $x_0$. We have, by the basic commutation relations in Kähler geometry,
$$
\Delta_{\omega'}S=\bar\partial^*\bar\partial S=-i\Lambda_{\omega'}\partial\bar\partial S,
$$
and thus, since acting with $\Lambda_{\omega'}$ on real $(1,1)$-forms amounts to takeing the trace with respect to $\omega'$,
$$
\Delta_{\omega'}S=-\operatorname{tr}_{\omega'}i\partial\bar\partial S=-\sum_{j,k=1}^n\Omega'_{kj}\,\frac{\partial^2 S}{\partial z_j\partial\bar z_k}.
$$
Now,
$$
\begin{aligned}
\frac{\partial^2 S}{\partial z_j\partial\bar z_k} &=\frac{\partial^2}{\partial z_j\partial\bar z_k}\sum_{l,m=1}^n\Omega'_{ml}\,\omega_{lm} \\
&=\sum_{l,m=1}^n\omega_{lm}\frac{\partial^2\Omega'_{ml}}{\partial z_j\partial\bar z_k}+
\Omega'_{ml}\frac{\partial^2\omega_{lm}}{\partial z_j\partial\bar z_k}
+\underbrace{\frac{\partial\Omega'_{ml}}{\partial z_j}\frac{\partial\omega_{lm}}{\partial\bar z_k}
+\frac{\partial\omega_{lm}}{\partial z_j}\frac{\partial\Omega'_{ml}}{\partial\bar z_k}}_{=O(|z|)}.
\end{aligned}
$$
At the end of the day, thanks to the choice of geodesic coordinates, the terms with only one derivative involved are $O(|z|)$'s and will disappear. Therefore, we only have to understand the summands with two derivatives of $\Omega'_{ml}$, and express them in terms of the $\omega'_{lm}$'s. In order to do this, call $H=(\omega'_{lm})$, so that $H^{-1}=(\Omega'_{lm})$ and observe that
\begin{equation}\label{partialH-1}
0\equiv\partial(HH^{-1})=\partial HH^{-1}+H\partial H^{-1},
\end{equation}
\begin{equation}\label{barpartialH-1}
0\equiv\bar\partial(H H^{-1})=\bar\partial H H^{-1}+H\bar\partial H^{-1},
\end{equation}
and
$$
0\equiv\partial\bar\partial(HH^{-1})=\partial\bar\partial H H^{-1}-\bar\partial H\wedge\partial H^{-1}+\partial H\wedge\bar\partial H^{-1}+H\partial\bar\partial H^{-1}.
$$
We obtain therefore the matrix identity

\begin{multline*}
\partial\bar\partial H^{-1}=-H^{-1}\partial\bar\partial H H^{-1}-H^{-1}\bar\partial H\wedge H^{-1}\partial H H^{-1} \\ +H^{-1}\partial H\wedge H^{-1}\bar\partial H H^{-1},
\end{multline*}

which gives us the following expression for the second derivatives of $\Omega'_{ml}$:

\begin{multline*}
\frac{\partial^2\Omega'_{ml}}{\partial z_j\partial\bar z_k}=-\sum_{a,b=1}^n\Omega'_{ma}\frac{\partial^2\omega'_{ab}}{\partial z_j\partial\bar z_k}\Omega'_{bl}
\\+\sum_{a,b,p,q=1}^n\Omega'_{mp}\frac{\partial\omega'_{pq}}{\partial\bar z_k}\Omega'_{qa}\frac{\partial\omega'_{ab}}{\partial z_j}\Omega'_{bl}
+\Omega'_{ma}\frac{\partial\omega'_{ab}}{\partial z_j}\Omega'_{bp}\frac{\partial\omega'_{pq}}{\partial\bar z_k}\Omega'_{ql}.
\end{multline*}

Thus, we get the following expression for the $\omega'$-Laplacian:

\begin{multline}\label{exprlapl}
\Delta_{\omega'}S=-\sum_{j,k,l,m=1}^n\Omega'_{kj}\biggl(\omega_{lm}\frac{\partial^2\Omega'_{ml}}{\partial z_j\partial\bar z_k}+
\Omega'_{ml}\frac{\partial^2\omega_{lm}}{\partial z_j\partial\bar z_k} + L_{jklm} \biggr) \\
= -\sum_{j,k,l,m=1}^n\Omega'_{kj}\Omega'_{ml}\frac{\partial^2\omega_{lm}}{\partial z_j\partial\bar z_k}
+\Omega'_{kj}\omega_{lm}\frac{\partial^2\Omega'_{ml}}{\partial z_j\partial\bar z_k}+\Omega'_{kj}\,L_{jklm}\\
=-\sum_{j,k,l,m=1}^n\Omega'_{kj}\Omega'_{ml}\frac{\partial^2\omega_{lm}}{\partial z_j\partial\bar z_k} +\Omega'_{kj}\,L_{jklm}\\
+\sum_{j,k,l,m,a,b=1}^n\omega_{lm}\Omega'_{kj}\Omega'_{ma}\Omega'_{bl}\frac{\partial^2\omega'_{ab}}{\partial z_j\partial\bar z_k} \\
-\sum_{j,k,l,m,a,b,p,q=1}^n\omega_{lm}\Omega'_{kj}\Omega'_{mp}\Omega'_{qa}\Omega'_{bl}\frac{\partial\omega'_{ab}}{\partial z_j}\frac{\partial\omega'_{pq}}{\partial\bar z_k}\\
-\sum_{j,k,l,m,a,b,p,q=1}^n\omega_{lm}\Omega'_{kj}\Omega'_{ma}\Omega'_{bp}\Omega'_{ql}\frac{\partial\omega'_{ab}}{\partial z_j}\frac{\partial\omega'_{pq}}{\partial\bar z_k}.
\end{multline}

Now, still denoting by $H$ the matrix $(\omega'_{lm})$, we recall the well-known formula to determine in local coordinates the Chern curvature of $\omega'$, namely
$$
\begin{aligned}
\Theta(T_X,\omega') & \simeq_{\textrm{loc}}\bar\partial\bigl(\bar H^{-1}\partial\bar H\bigr) \\
&=\bar\partial\bar H^{-1}\wedge\partial \bar H+\bar H^{-1}\bar\partial\partial\bar H \\
&=-\bar H^{-1}\bar\partial\bar H\bar H^{-1}\wedge\partial\bar H^{-1}+\bar H^{-1}\bar\partial\partial\bar H,
\end{aligned}
$$
where the last equality is obtained by using formula (\ref{barpartialH-1}). So, if we write in these coordinates
$$
\Theta(T_X,\omega')=\sum_{j,k,l,m}c'_{jklm}\,dz_j\wedge d\bar z_k\otimes\biggl(\frac\partial{\partial z_l}\biggr)^*\otimes\frac\partial{\partial\bar z_m},
$$
we obtain the following expression for the coefficients of the Chern curvature tensor:
\begin{equation}\label{c'jklm}
c'_{jkal}=-\sum_{b=1}^n\Omega'_{bl}\frac{\partial^2\omega'_{ab}}{\partial z_j\partial\bar z_k}
+\sum_{b,p,q=1}^n\Omega'_{pl}\Omega'_{bq}\frac{\partial\omega'_{ab}}{\partial z_j}\frac{\partial\omega'_{qp}}{\partial\bar z_k}.
\end{equation}
We can now use the above identity (\ref{c'jklm}) to replace in the right hand side of formula (\ref{exprlapl}) the summand  
$$
\sum_{b=1}^n\Omega'_{bl}\frac{\partial^2\omega'_{ab}}{\partial z_j\partial\bar z_k}
$$
with
$$
-c'_{jkal}+\sum_{b,p,q=1}^n\Omega'_{pl}\Omega'_{bq}\frac{\partial\omega'_{ab}}{\partial z_j}\frac{\partial\omega'_{qp}}{\partial\bar z_k}.
$$
With this substitution, we obtain
\begin{multline}\label{exprlaplsemifinal}
\Delta_{\omega'}S
=-\sum_{j,k,l,m=1}^n\Omega'_{kj}\Omega'_{ml}\frac{\partial^2\omega_{lm}}{\partial z_j\partial\bar z_k} +\Omega'_{kj}\,L_{jklm}\\
+\sum_{j,k,l,m,a=1}^n\omega_{lm}\Omega'_{kj}\Omega'_{ma}\biggl(-c'_{jkal}+\sum_{b,p,q=1}^n\Omega'_{pl}\Omega'_{bq}\frac{\partial\omega'_{ab}}{\partial z_j}\frac{\partial\omega'_{qp}}{\partial\bar z_k}\biggr) \\
-\sum_{j,k,l,m,a,b,p,q=1}^n\omega_{lm}\Omega'_{kj}\Omega'_{mp}\Omega'_{qa}\Omega'_{bl}\frac{\partial\omega'_{ab}}{\partial z_j}\frac{\partial\omega'_{pq}}{\partial\bar z_k}\\
-\sum_{j,k,l,m,a,b,p,q=1}^n\omega_{lm}\Omega'_{kj}\Omega'_{ma}\Omega'_{bp}\Omega'_{ql}\frac{\partial\omega'_{ab}}{\partial z_j}\frac{\partial\omega'_{pq}}{\partial\bar z_k} \\
=-\sum_{j,k,l,m=1}^n\Omega'_{kj}\Omega'_{ml}\frac{\partial^2\omega_{lm}}{\partial z_j\partial\bar z_k} +\Omega'_{kj}\,L_{jklm}\\
-\sum_{j,k,l,m,a=1}^n\omega_{lm}\Omega'_{kj}\Omega'_{ma}c'_{jkal} \\
-\sum_{j,k,l,m,a,b,p,q=1}^n\omega_{lm}\Omega'_{kj}\Omega'_{mp}\Omega'_{qa}\Omega'_{bl}\frac{\partial\omega'_{ab}}{\partial z_j}\frac{\partial\omega'_{pq}}{\partial\bar z_k}.
\end{multline}

Now, since $(\partial/\partial z_1,\dots,\partial/\partial z_n)$ is merely $\omega'$-orthogonal but not necessarily $\omega'$-unitary at $x_0$, the Kähler symmetries of the coefficients $c'_{jklm}$'s at $x_0$ read
$$
c'_{jklm}\sqrt{\lambda_l}\sqrt{\lambda_m}=c'_{lmjk}\sqrt{\lambda_j}\sqrt{\lambda_k}.
$$ 
In particular, 
$$
c'_{jjll}\lambda_l=c'_{lljj}\lambda_j.
$$
We are now in a good position to conclude the proof of the lemma. Indeed, evaluating (\ref{exprlaplsemifinal}) at the point $x_0$ with our initial choice of coordinates gives
\begin{equation}\label{exprlaplfinal}
\begin{aligned}
\Delta_{\omega'}S(x_0)
&=\sum_{j,l=1}^n\frac{c_{jjll}}{\lambda_j\lambda_l}
-\sum_{j,l=1}^n\underbrace{\frac{c'_{jjll}}{\lambda_j\lambda_l}}_{=\frac{c'_{lljj}}{(\lambda_l)^2}} 
-\sum_{j,l,a=1}^n\frac 1{\lambda_j(\lambda_l)^2\lambda_a}\frac{\partial\omega'_{al}}{\partial z_j}\frac{\partial\omega'_{la}}{\partial\bar z_j} \\
&=\sum_{j,l=1}^n\frac{c_{jjll}}{\lambda_j\lambda_l}-\sum_{l=1}^n\frac{\rho'_{ll}}{(\lambda_l)^2}
-\sum_{j,l,a=1}^n\frac {|\partial\omega'_{al}/\partial z_j(x_0)|^2}{\lambda_j(\lambda_l)^2\lambda_a}.
\end{aligned}
\end{equation}
\end{proof}

Our next task will be to estimate the three summands appearing on the right hand side of the differential equality of the above lemma. We begin with the term involving the Ricci curvature of $\omega'$. Recall that the we are supposing that
$$
\operatorname{Ric}(\omega')\ge-\lambda\omega'+\mu\omega.
$$
\begin{lemma}\label{term1}
At the point $x_0\in X$, we have
$$
\sum_{l=1}^n\frac{\rho'_{ll}}{(\lambda_l)^2}\ge 2\pi\biggl(-\lambda S+\frac\mu n S^2\biggr).
$$
\end{lemma}

\begin{proof}
The hypothesis on the Ricci curvature of $\omega'$, when red at the point $x_0$ with our choice of coordinates, gives
$$
\rho'_{ll}\ge 2\pi(-\lambda\lambda_l+\mu).
$$
Thus, we get
$$
\sum_{l=1}^n\frac{\rho'_{ll}}{(\lambda_l)^2}\ge -2\pi\lambda\sum_{l=1}^n\frac{1}{\lambda_l}+2\pi\mu\sum_{l=1}^n\frac{1}{(\lambda_l)^2}.
$$
Now, since the $\lambda_l$'s are the eigenvalues of $\omega'$ with respect to $\omega$, the eigenvalues of omega with respect to $\omega'$ are $1/\lambda_l$, $l=1,\dots,n$, and therefore $S=\sum_{l=1}^n1/\lambda_l$. Moreover, by the standard inequality between $1$-norm and $2$-norm of vectors in $\mathbb R^n$, we have $\sum_{l=1}^n1/(\lambda_l)^2\ge 1/n\bigl(\sum_{l=1}^n1/\lambda_l\bigr)^2$. We finally obtain
$$
\sum_{l=1}^n\frac{\rho'_{ll}}{(\lambda_l)^2}\ge 2\pi\biggl(-\lambda S+\frac\mu n S^2\biggr).
$$
\end{proof}

Now, we treat the term with the first order derivatives of the metric $\omega'$. In doing this, we have to keep in mind that, at the end of the day, we want to estimate $\Delta_{\omega'}\log S$. This Laplacian is given in coordinates by
$$
\begin{aligned}
\Delta_{\omega'}\log S&=-\sum_{j,k=1}^n\Omega'_{kj}\underbrace{\frac{\partial^2\log S}{\partial z_j\partial\bar z_k}}_{=\frac{\partial}{\partial z_j}\biggl(\frac 1S \frac{\partial S}{\partial\bar z_k}\biggr)=-\frac 1{S^2}\frac{\partial S}{\partial z_j}\frac{\partial S}{\partial\bar z_k}+\frac 1S\frac{\partial^2 S}{\partial z_j\partial\bar z_k}}\\
&=\frac 1S\Delta_{\omega'}S+\frac 1{S^2}\sum_{j,k=1}^n\Omega'_{kj}\frac{\partial S}{\partial z_j}\frac{\partial S}{\partial\bar z_k}.
\end{aligned}
$$
Once computed at $x_0$, we have
\begin{equation}\label{laplacianlog}
\Delta_{\omega'}\log S(x_0)=\frac 1{S(x_0)}\Delta_{\omega'}S(x_0)+\frac 1{S(x_0)^2}\sum_{j=1}^n\frac 1{\lambda_j}\biggl|\frac{\partial S}{\partial z_j}(x_0)\biggr|^2.
\end{equation}
What we want to do in the lemma below is then to try to express these first order derivatives in terms of first order derivatives of $S$.

\begin{lemma}\label{term2}
At the point $x_0\in X$, we have
$$
\sum_{j,l,a=1}^n\frac {\bigl|\partial\omega'_{al}/\partial z_j\bigr|^2}{\lambda_j(\lambda_l)^2\lambda_a}\ge\frac 1{S(x_0)}\sum_{j=1}^n\frac 1{\lambda_j}\biggl|\frac{\partial S}{\partial z_j}(x_0)\biggr|^2.
$$
\end{lemma}

\begin{proof}
Since the sum we are dealing with is made up of non negative terms, we have by plain minoration
$$
\sum_{j,l,a=1}^n\frac {\bigl|\partial\omega'_{al}/\partial z_j\bigr|^2}{\lambda_j(\lambda_l)^2\lambda_a}\ge
\sum_{j,l=1}^n\frac {\bigl|\partial\omega'_{ll}/\partial z_j\bigr|^2}{\lambda_j(\lambda_l)^3}.
$$
Now, let us compute $\partial S/\partial z_j$ at $x_0$. We have
$$
\begin{aligned}
\frac{\partial S}{\partial z_j}(x_0)&= \frac{\partial}{\partial z_j}\sum_{l,m=1}^n\Omega'_{ml}\omega_{lm}\biggr|_{x_0} \\
&= \sum_{l,m=1}^n\frac{\partial\Omega'_{ml}}{\partial z_j}\omega_{lm}+\Omega'_{ml}\frac{\partial\omega_{lm}}{\partial z_j}\biggr|_{x_0}=\sum_{l=1}^n\frac{\partial\Omega'_{ll}}{\partial z_j}(x_0).
\end{aligned}
$$
Now, we use the identity (\ref{partialH-1}) to replace $\partial\Omega'_{ll}/\partial z_j(x_0)$ with
$$
-\sum_{a,b=1}^n\Omega'_{la}\frac{\partial\omega'_{ab}}{\partial z_j}\Omega'_{bl}\biggr|_{x_0}=-\frac 1{(\lambda_l)^2}\frac{\partial\omega'_{ll}}{\partial z_j}(x_0).
$$
Thus, we obtain
$$
\frac{\partial S}{\partial z_j}(x_0)=-\sum_{l=1}^n\frac 1{(\lambda_l)^2}\frac{\partial\omega'_{ll}}{\partial z_j}(x_0).
$$
Now, inspired by (\ref{laplacianlog}), we compute
$$
\begin{aligned}
\sum_{j=1}^n\frac 1{\lambda_j}\biggl|\frac{\partial S}{\partial z_j}(x_0)\biggr|^2 &=
\sum_{j=1}^n\frac 1{\lambda_j}\left|\sum_{l=1}^n\frac 1{(\lambda_l)^2}\frac{\partial\omega'_{ll}}{\partial z_j}(x_0)\right|^2 \\
&=\sum_{j=1}^n\frac 1{\lambda_j}\left|\sum_{l=1}^n\frac 1{(\lambda_l)^{1/2}}\frac{\partial\omega'_{ll}/\partial z_j(x_0)}{(\lambda_l)^{3/2}}\right|^2 \\
&\le \sum_{j=1}^n\frac 1{\lambda_j}\sum_{k=1}^n\frac 1{\lambda_k}\sum_{l=1}^n\frac{\bigl|\partial\omega'_{ll}/\partial z_j(x_0)\bigr|^2}{(\lambda_l)^{3}} \\
&=S(x_0)\sum_{l,j=1}^n\frac {\bigl|\partial\omega'_{ll}/\partial z_j(x_0)\bigr|^2}{\lambda_j(\lambda_l)^3},
\end{aligned}
$$
where the inequality is given by Cauchy-Schwarz. The lemma follows.
\end{proof}

Finally, we estimate the term involving the curvature of $\omega$, using the hypothesis on the negativity of the holomorphic sectional curvature.

\begin{lemma}\label{term3}
At the point $x_0\in X$, we have
$$
\sum_{j,l=1}^n\frac{c_{jjll}}{\lambda_j\lambda_l}\le -\frac{\kappa(n+1)}{2n}S^2.
$$
\end{lemma}

Classically this term has been bounded in terms of a uniform bound on the holomorphic bisectional curvature of $\omega$. In order to prove this lemma, we need to be able to transform an information on the sum of holomorphic bisectional curvature type terms into an estimate using holomorphic sectional curvature only.
Next proposition in hermitian linear algebra is the key point to do that. It is due to Royden.

\begin{proposition}[Royden \cite{Roy80}]\label{royden}
Let $\xi_1,\dots,\xi_\nu$ be mutually orthogonal (but not necessarily unitary) non-zero vectors of a hermitian vector space $(V,h)$. Suppose that $\Theta(\xi,\eta,\zeta,\omega)$ is a symmetric \lq\lq bi-hermitian\rq\rq{} form, \textsl{i.e.} $\Theta$ is sesquilinear in the first two and last two variables and has the same pointwise properties as those of the (contraction with the metric of the) Chern curvature of a Kähler metric. Suppose also that there exists a real constant $K$ such that for all $\xi\in V$ one has
$$
\Theta(\xi,\xi,\xi,\xi)\le K\,||\xi||^4_h.
$$
Then, 
$$
\sum_{\alpha,\beta}\Theta(\xi_\alpha,\xi_\alpha,\xi_\beta,\xi_\beta)\le\frac 12\,K\,\left(\biggl(\sum_\alpha ||\xi_\alpha||^2_h\biggr)^2+\sum_\alpha||\xi_\alpha||^4_h\right).
$$
Moreover, if $K\le 0$, then
$$
\sum_{\alpha,\beta}\Theta(\xi_\alpha,\xi_\alpha,\xi_\beta,\xi_\beta)\le\frac{\nu+1}{2\nu}\,K\,\biggl(\sum_\alpha||\xi_\alpha||^2_h\biggr)^2.
$$
\end{proposition}

We shall use this proposition with 
$$
\bigl\langle\Theta(T_X,\omega)(\bullet,\bar\bullet)\cdot\bullet,\bullet\bigr\rangle_\omega
$$
as the symmetric \lq\lq bi-hermitian\rq\rq{} form on $T_{X,x_0}$ in the statement. In terms of holomorphic bisectional curvature it can be rephrased as follows, when $\nu=n=\dim X$.

\noindent
\emph{Suppose that a Kähler metric $\omega$ has negative holomorphic sectional curvature at the point $x_0$, bounded above by a negative constant $-\kappa$. Then, if $\xi_1,\dots,\xi_n$ is a $\omega$-orthogonal basis for $T_{X,x_0}$ we have
$$
\sum_{\alpha,\beta=1}^n ||\xi_\alpha||^2_\omega ||\xi_\beta||^2_\omega\,\operatorname{HBC}_\omega(\xi_\alpha,\xi_\beta)\le-\frac{\kappa(n+1)}{2n}\biggl(\sum_{\alpha=1}^n||\xi_\alpha||^2_\omega\biggr)^2.
$$}

Here is the proof.

\begin{proof}
Realize $\mathbb Z_4$ as the group of 4th roots of unity and set, for a vector $A=(\epsilon_1,\dots,\epsilon_\nu)\in\mathbb Z^\nu_4$,
$$
\xi_A=\sum_\alpha\epsilon_\alpha\,\xi_\alpha.
$$
Then, by orthogonality, $||\xi_A||^2_h=\sum_\alpha ||\xi_\alpha||^2_h$, and thus by hypothesis
$$
\Theta(\xi_A,\xi_A,\xi_A,\xi_A)\le K\,||\xi_A||^4_h= K\,\biggl(\sum_\alpha||\xi_\alpha||^2_h\biggr)^2.
$$
Now, we take the sum over all possible $A\in\mathbb Z^\nu_4$ and get
$$
\begin{aligned}
K\,\biggl(\sum_\alpha||\xi_\alpha||^2_h\biggr)^2 &\ge\frac 1{4^\nu}\sum_{A}\Theta(\xi_A,\xi_A,\xi_A,\xi_A) \\
&=\frac 1{4^\nu}\sum_{A}\sum_{\alpha,\beta,\gamma,\delta}\epsilon_\alpha\bar\epsilon_\beta\epsilon_\gamma\bar\epsilon_\delta\,\Theta(\xi_\alpha,\xi_\beta,\xi_\gamma,\xi_\delta) \\
&=\frac 1{4^\nu}\sum_{\alpha,\beta,\gamma,\delta}\sum_{A}\frac{\epsilon_\alpha\epsilon_\gamma}{\epsilon_\beta\epsilon_\delta}\,\Theta(\xi_\alpha,\xi_\beta,\xi_\gamma,\xi_\delta).
\end{aligned}
$$
Now, fix a $4$-tuple $(\alpha,\beta,\gamma,\delta)$. We claim that only the terms with $\alpha=\beta$ and $\gamma=\delta$ or $\alpha=\delta$ and $\beta=\gamma$ can survive after summing over all $A$. Thus, we are left only with the following terms
\begin{multline*}
\frac 1{4^\nu}\sum_{\alpha,\beta,\gamma,\delta}\sum_{A}\frac{\epsilon_\alpha\epsilon_\gamma}{\epsilon_\beta\epsilon_\delta}\,\Theta(\xi_\alpha,\xi_\beta,\xi_\gamma,\xi_\delta) \\
= \sum_\alpha \Theta(\xi_\alpha,\xi_\alpha,\xi_\alpha,\xi_\alpha)+\sum_{\alpha\ne\gamma}\Theta(\xi_\alpha,\xi_\alpha,\xi_\gamma,\xi_\gamma)+\Theta(\xi_\alpha,\xi_\gamma,\xi_\gamma,\xi_\alpha).
\end{multline*}
The claim is straightforwardly verified, since for all the other terms, for each $A\in\mathbb Z^\nu_4$ one can find an $A'=(\epsilon_1',\dots,\epsilon_\nu')\in\mathbb Z^\nu_4$ such that $\frac{\epsilon_\alpha\epsilon_\gamma}{\epsilon_\beta\epsilon_\delta}=-\frac{\epsilon_\alpha'\epsilon_\gamma'}{\epsilon_\beta'\epsilon_\delta'}$.

Now, by symmetry of $\Theta$, adding $\sum_\alpha \Theta(\xi_\alpha,\xi_\alpha,\xi_\alpha,\xi_\alpha)$ to both side and using the upper bound as in the hypotheses, we get
$$
2\sum_{\alpha,\gamma}\Theta(\xi_\alpha,\xi_\alpha,\xi_\gamma,\xi_\gamma)\le K\,\left(\biggl(\sum_\alpha ||\xi_\alpha||^2_h\biggr)^2+\sum_\alpha||\xi_\alpha||^4_h\right).
$$
To end the proof, observe that applying the Cauchy--Schwarz inequality in $\mathbb R^\nu$ to the vectors $(||\xi_1||^2_h,\dots,||\xi_\nu||^2_h)$ and $(1,\dots,1)$, we have
$$
\biggl(\sum_\alpha ||\xi_\alpha||^2_h\biggr)^2\le \nu\,\sum_\alpha||\xi_\alpha||^4_h,
$$
so that if $K\le 0$, then 
$$
K\,\sum_\alpha||\xi_\alpha||^4_h\le \frac K\nu\,\biggl(\sum_\alpha ||\xi_\alpha||^2_h\biggr)^2,
$$
and thus
$$
\sum_{\alpha,\gamma}\Theta(\xi_\alpha,\xi_\alpha,\xi_\gamma,\xi_\gamma)\le\frac{\nu+1}{2\nu} K\,\biggl(\sum_\alpha ||\xi_\alpha||^2_h\biggr)^2,
$$
as desired.
\end{proof}

We are now ready to give a

\begin{proof}[Proof of Lemma \ref{term3}]
Set 
$$
\xi_j:=\frac 1{\sqrt{\lambda_j}}\frac \partial{\partial z_j},\quad j=1,\dots n,
$$
so that $\xi_1,\dots,\xi_n$ is a $\omega$-orthogonal basis for $T_{X,x_0}$. Now, it suffices to observe that 
$$
\begin{aligned}
\frac{c_{jjll}}{\lambda_j\lambda_l}&=\bigl\langle\Theta(T_X,\omega)(\xi_j,\bar\xi_j)\cdot\xi_l,\xi_l\bigr\rangle_\omega \\
&= ||\xi_j||^2_\omega ||\xi_l||^2_\omega\,\operatorname{HBC}_\omega(\xi_j,\xi_l).
\end{aligned}
$$
Now take the sum over all $j,l=1,\dots,n$, to obtain
$$
\sum_{j,l=1}^n\frac{c_{jjll}}{\lambda_j\lambda_l} \le -\frac{\kappa(n+1)}{2n}\biggl(\sum_{\alpha=1}^n\frac 1{\lambda_j}\biggr)^2 \\
=-\frac{\kappa(n+1)}{2n}S^2.
$$
\end{proof}

Now, to conclude the proof of Claim \ref{step2}, \textsl{i.e.} to show inequality (\ref{diffineq}), we just put together the three estimates of the above lemmata, and plug them into formula (\ref{laplacianlog}). We get:
$$
\begin{aligned}
-\Delta_{\omega'}\log S(x_0)&=-\frac 1{S(x_0)}\Delta_{\omega'}S(x_0)-\frac 1{S(x_0)^2}\sum_{j=1}^n\frac 1{\lambda_j}\biggl|\frac{\partial S}{\partial z_j}(x_0)\biggr|^2 \\
&=\frac 1{S(x_0)}\left(\sum_{l=1}^n\frac{\rho'_{ll}}{(\lambda_l)^2}
+\sum_{j,l,a=1}^n\frac {\bigl|\partial\omega'_{al}/\partial z_j(x_0)\bigr|^2}{\lambda_j(\lambda_l)^2\lambda_a}-\sum_{j,l=1}^n\frac{c_{jjll}}{\lambda_j\lambda_l}\right)\\
&\qquad-\frac 1{S(x_0)^2}\sum_{j=1}^n\frac 1{\lambda_j}\biggl|\frac{\partial S}{\partial z_j}(x_0)\biggr|^2 \\
& \ge\frac 1{S(x_0)}\left(2\pi\biggl(-\lambda S(x_0)+\frac\mu n S(x_0)^2\biggr)\right. \\
&\left.\qquad+\frac 1{S(x_0)}\sum_{j=1}^n\frac 1{\lambda_j}\biggl|\frac{\partial S}{\partial z_j}(x_0)\biggr|^2+\frac{\kappa(n+1)}{2n}S(x_0)^2\right)\\
&\qquad\qquad-\frac 1{S(x_0)^2}\sum_{j=1}^n\frac 1{\lambda_j}\biggl|\frac{\partial S}{\partial z_j}(x_0)\biggr|^2 \\
&=\biggl(\frac{\kappa(n+1)}{2n}+2\pi\frac\mu n\biggr)S(x_0)-2\pi\lambda,
\end{aligned}
$$
as desired.
\end{proof}

\begin{proof}[Proof of Claim \ref{step3}]
We want to show that the limit
$$
\lim_{\varepsilon\to 0^+}\int_X\omega_\varepsilon^n=\lim_{\varepsilon\to 0^+}\int_X e^{u_\varepsilon}\,\omega^n
$$
is strictly positive.
The first observation is that the functions $u_\varepsilon$ are all $\omega'$-plurisubharmonic for some fixed Kähler form $\omega'$ and $\varepsilon>0$ small enough. For, let $\ell>0$ be such that $\ell\omega-\Ric_\omega$ is positive and call $\omega'=\ell\omega-\Ric_\omega$. Thus, for all $0<\varepsilon<\ell$, one has 
$$
0<\varepsilon\omega-\Ric_\omega+i\partial\bar\partial u_\varepsilon<\ell\omega-\Ric_\omega+i\partial\bar\partial u_\varepsilon=\omega'+i\partial\bar\partial u_\varepsilon.
$$
Therefore, since the $u_\varepsilon$'s are all uniformly bounded from above, by \cite[Proposition 2.6]{GZ05}, either $\{u_\varepsilon\}$ converges uniformly to $-\infty$ on $X$ or it is relatively compact in $L^1(X)$. Suppose for a moment that we are in the second case. Then, there exists a subsequence $\{u_{\varepsilon_k}\}$ converging in $L^1(X)$ and moreover the limit coincides \textsl{a.e.} with a uniquely determined $\omega'$-plurisubharmonic function $u$. Up to passing to a further subsequence, we can also suppose that $u_{\varepsilon_k}$ converges pointwise \textsl{a.e.} to $u$. But then, $e^{u_{\varepsilon_k}}\to e^u$ pointwise \textsl{a.e.} on $X$. On the other hand, by Claim \ref{step1}, we have $e^{u_{\varepsilon_k}}\le e^C$ so that, by dominated convergence, we also have $L^1(X)$-convergence and
$$
\lim_{k\to\infty}\int_X e^{u_{\varepsilon_k}}\,\omega^n=\int_X e^u\omega^n>0.
$$
The upshot is that what we need to prove is that $\{u_\varepsilon\}$ does not converge uniformly to $-\infty$ on $X$. We shall thus provide a lower bound for $\sup_X u_\varepsilon$, as follows.

Recall that we defined the smooth positive function $S_\varepsilon$ on $X$ by
$$
\omega\wedge\omega_\varepsilon^{n-1}=\frac{S_\varepsilon}n\,\omega_\varepsilon^{n}.
$$
Now, set $T_\varepsilon=\log S_\varepsilon$. In other words, $T_\varepsilon$ is the logarithm of the trace of $\omega$ with respect to $\omega_\varepsilon$.

\begin{lemma}\label{lem:sge-m}
The function $T_\varepsilon$ satisfies the following inequality:
$$
T_\varepsilon> -\frac{u_\varepsilon}n.
$$
\end{lemma}

\begin{proof}
Let $0<\lambda_{1}\le\cdots\le\lambda_n$ be the eigenvalues of $\omega_\varepsilon$ with respect to $\omega$, so that $0<1/\lambda_{n}\le\cdots\le 1/\lambda_1$ are the eigenvalues of $\omega$ with respect to $\omega_\varepsilon$. Then,
$$
e^{T_\varepsilon}=\operatorname{tr}_{\omega_\varepsilon}\omega=\frac 1{\lambda_1}+\cdots+\frac 1{\lambda_n}>\frac 1{\lambda_1}.
$$
Thus, $e^{-T_\varepsilon}<\lambda_1$ so that $e^{-nT_\varepsilon}<(\lambda_1)^n\le\lambda_1\cdots\lambda_n$.
But, $e^{u_\varepsilon}\omega^n=\omega_\varepsilon^n=\lambda_1\cdots\lambda_n\,\omega^n$, and so we get $
e^{-nT_\varepsilon}<e^{u_\varepsilon}$, or, in other words,
$$
T_\varepsilon>-\frac{u_\varepsilon}n.
$$
\end{proof}

As announced, we now use the integral inequality (\ref{intineq}). We write it as follows:
$$
\frac{(n+1)\kappa}{2n}\int_X e^{T_\varepsilon+u_\varepsilon}\,\omega^n\le 2\pi\int_X e^{u_\varepsilon}\,\omega^n.
$$
Next, if we define $C_\varepsilon:=\inf_X e^{-u_\varepsilon/n}$, by Lemma \ref{lem:sge-m} we have that $e^{T_\varepsilon}> C_\varepsilon$, and thus

\begin{equation}
C_\varepsilon\,\frac{(n+1)\kappa}{2n}\int_X e^{u_\varepsilon}\,\omega^n\le 2\pi\int_X e^{u_\varepsilon}\,\omega^n.
\end{equation}
But then, we obtain that 
$$
C_\varepsilon\,\frac{(n+1)\kappa}{2n}\le 2\pi,
$$
\textsl{i.e.}
$$
\inf_X e^{-u_\varepsilon/n}\le\frac{4\pi n}{(n+1)\kappa},
$$
so that
$$
\sup_X u_\varepsilon\ge -n\log\frac{4\pi n}{(n+1)\kappa}
$$ 
is the desired lower bound.
\end{proof}

\section{The Kähler case, and the quasi-negative holomorphic sectional curvature case}

As largely revealed in advance, the Wu--Yau theorem also holds for compact Kähler manifolds under the same curvature assumptions. This is due to Tosatti and Yang, on the lines of Wu-Yau's proof. One can also relax, still in the Kähler case, the negativity hypothesis on the curvature to quasi-negativity, as explained in Section 3.

Let us briefly explain how to obtain such results.

\subsection{The Kähler case}

In order to adapt the proof explained in the preceding section to the compact Kähler case, the main point is to show that a compact Kähler manifold endowed with a Kähler metric whose holomorphic sectional curvature is negative has nef canonical bundle. Indeed, this is obtained in the projective case by Mori's theorem, which is unknown for compact Kähler manifolds. The nefness of the canonical bundle still holds more generally under the (more natural) assumption of non positivity of the holomorphic sectional curvature.

\begin{theorem}[{\cite[Theorem 1.1]{TY15}}]
Let $(X,\omega)$ be a compact Kähler manifold with non positive holomorphic sectional curvature. Then the canonical bundle $K_X$ is nef.
\end{theorem}

\begin{proof}[Sketch of the proof]
Suppose by contradiction that $K_X$ is not nef, that is $-c_1(X)$ is not in the closure of the Kähler cone. Since $t[\omega]-c_1(X)$ becomes a Kähler class for $t\gg 0$, one can select the time $\varepsilon_0$ when $t[\omega]-c_1(X)$ cuts the boundary of the Kähler cone, that is $\varepsilon_0[\omega]-c_1(X)$ is a nef but not Kähler class. Thus, for every $\varepsilon>0$ one gets a Kähler class $(\varepsilon+\varepsilon_0)[\omega]-c_1(X)$. Therefore, by $\partial\bar\partial$-lemma and Yau's solution of Calabi's conjecture, for every $\varepsilon>0$ we obtain a Kähler metric $\omega_\varepsilon\in(\varepsilon+\varepsilon_0)[\omega]-c_1(X)$ such that
$$
\omega_\varepsilon=(\varepsilon+\varepsilon_0)\omega-\Ric_\omega+\frac i{2\pi}\partial\bar\partial u_\varepsilon
$$
and
$$
\omega_\varepsilon^n=e^{u_\varepsilon}\omega^n.
$$
For this metric, we have
$$
\Ric_{\omega_\varepsilon}=-\omega_\varepsilon+(\varepsilon+\varepsilon_0)\omega.
$$
Now, precisely as in Claim \ref{step1} and \ref{step2}, we get a uniform upper bound on the one hand for
$$
\sup_X e^{u_\varepsilon}\le C,
$$
and, on the other hand, for
$$
-\Delta_{\omega_{\varepsilon}}\log S_\varepsilon\ge\biggl(\frac{\kappa(n+1)}{2n}+2\pi\frac\mu n\biggr)S_\varepsilon-2\pi\lambda,
$$
where $S_\varepsilon=\operatorname{tr}_{\omega_\varepsilon}\omega$, and here $\kappa=0$, $\lambda=1$, and $\mu=\varepsilon+\varepsilon_0$. By the maximum principle we obtain
$$
\sup_X\operatorname{tr}_{\omega_\varepsilon}\omega\le\frac n{\varepsilon+\varepsilon_0},
$$
which is uniformly bounded as $\varepsilon$ approaches to zero, since $\Delta_{\omega_{\varepsilon}}\log S_\varepsilon$ is non negative at a point where $S_\varepsilon$ (and hence $\log S_\varepsilon$) achieves a (local) maximum. Now, the elementary inequality
$$
\operatorname{tr}_{\omega}\omega_\varepsilon\le\frac 1{(n-1)!}\bigl(\operatorname{tr}_{\omega_\varepsilon}\omega\bigr)^{n-1}\underbrace{\frac{\omega_\varepsilon^n}{\omega^n}}_{=e^{u_\varepsilon}},
$$
enables us to conclude that also $\operatorname{tr}_{\omega}\omega_\varepsilon$ is uniformly bounded from above, so that we have
$$
B^{-1}\,\omega\le\omega_\varepsilon\le B\,\omega,
$$
for some positive constant $B$. Beside this control for $\omega_\varepsilon$, it is also possible to work out uniform $C^k$-estimates for all $k\ge 0$ ---as explained in \cite[pp. 577--578]{TY15}--- and these together with the Ascoli--Arzelà Theorem and a diagonal argument allow to obtain the existence of a sequence $\varepsilon_k\to 0$ such that $\omega_{\varepsilon_k}$ converges smoothly to a Kähler metric $\omega_0$ which of course satisfies $[\omega_0] = \varepsilon_0\,[\omega]-c_1(X)$. But this is a contradiction, since we were supposing that $\varepsilon_0\,[\omega]-c_1(X)$ is a nef but not Kähler class, and hence it cannot contain any Kähler metric.
\end{proof}

Once the nefness of the canonical class is known also in the Kähler setting the proof proceeds in the same way of the projective case, with a small further argument at the end. Indeed, after proving Claim \ref{step3}, we know that $c_1(K_X)^n>0$. Then, by \cite[Theorem 0.5]{DP04}, we deduce that $K_X$ is big. In particular, carrying a big line bundle, $X$ is Moishezon. Since $X$ is Kähler and Moishezon, by Moishezon’s theorem $X$ is projective. We have thus finally landed in the projective world, where we can apply Lemma \ref{lem:ratcurv} and conclude the proof of the ampleness.

\subsection{The Kähler quasi-negative case}

Relaxing further the hypotheses, we finally come to the case of a compact Kähler manifold supporting a Kähler metric whose holomorphic sectional curvature is quasi-negative.

Thus, the first problem to be faced is that we do not dispose anymore of a negative uniform upper bound for the holomorphic sectional curvature. Such a bound is here replaced by the continuous function on $X$:
$$
\begin{aligned}
\kappa\colon & X\to\mathbb R \\
& x\mapsto -\max_{v\in T_{X,x}\setminus\{0\}}\HSC_\omega(x,[v]).
\end{aligned}
$$
The quasi-negativity of the holomorphic sectional curvature of Theorem \ref{thm:main} translates in $\kappa\ge 0$ and $\kappa(x_0)>0$ for some $x_0\in X$.

As we saw, for every $\varepsilon>0$ we have the following crucial inequality which makes the holomorphic sectional curvature enter into the picture:
\begin{equation}\label{eq:maindiffineq}
\Delta_{\omega_\varepsilon}T_\varepsilon(x)\ge\biggl(\frac{n+1}{2n}\kappa(x)+\frac\varepsilon n\biggr)e^{T_\varepsilon(x)}-1.
\end{equation}
The inequality being of pointwise nature, it still holds with using the continuous function $\kappa$.
Set $M(x)=\frac{n+1}{2n}\kappa(x)$. By plain minoration of the right hand side, we obtain that the $T_\varepsilon$'s satisfy the following differential inequality:
\begin{equation}\label{eq:maindiffineq'}
\Delta_{\omega_\varepsilon}T_\varepsilon(x)\ge M(x)\,e^{T_\varepsilon(x)}-1.
\end{equation}

For each $\varepsilon>0$, as before, integrate (\ref{eq:maindiffineq'}) over $X$ using the volume form associated to $\omega_\varepsilon$, to get
$$
0=\int_X\Delta_{\omega_\varepsilon}T_\varepsilon\,\omega_\varepsilon^n\ge\int_X\bigl(Me^{T_\varepsilon}-1\bigr)\,\omega_\varepsilon^n.
$$
We obtain therefore the following integral inequality:
$$
\int_X Me^{T_\varepsilon}e^{u_\varepsilon}\,\omega^n\le\int_X e^{u_\varepsilon}\,\omega^n,
$$
and setting $v_\varepsilon=u_\varepsilon-\sup_X u_\varepsilon$ one has
$$
\int_X Me^{T_\varepsilon}e^{v_\varepsilon}\,\omega^n\le\int_X e^{v_\varepsilon}\,\omega^n.
$$
Next, if we define $C_\varepsilon:=\inf_X e^{-u_\varepsilon/n}$, we have that $e^{T_\varepsilon}> C_\varepsilon$, and
\begin{equation}\label{eq:final}
C_\varepsilon\int_X Me^{v_\varepsilon}\,\omega^n\le\int_X e^{v_\varepsilon}\,\omega^n.
\end{equation}
Moreover, recall that we are assuming by contradiction that $C_\varepsilon\to +\infty$ as $\varepsilon\to 0^+$. 

Now, the same reasoning made during the proof of Claim \ref{step3} tells us that there exists a subsequence $\{v_{\varepsilon_k}\}$ of $\{v_\varepsilon\}$ converging in $L^1(X)$ and moreover the limit coincides \textsl{a.e.} with a uniquely determined $\omega'$-plurisubharmonic function $v$. Indeed, the case where $\{v_\varepsilon\}$ converges uniformly to $-\infty$ is not possible here since the supremum of the $v_\varepsilon$'s is fixed and equal to $0$. Again, up to pass to a further subsequence, we can also suppose that $v_{\varepsilon_k}$ converges pointwise \textsl{a.e.} to $v$. But then, $e^{v_{\varepsilon_k}}\to e^v$ pointwise \textsl{a.e.} on $X$. On the other hand, we have $e^{v_{\varepsilon_k}}\le 1$ so that, by dominated convergence, we also have $L^1(X)$-convergence and therefore
$$
\lim_{k\to\infty}\int_X e^{v_{\varepsilon_k}}\,\omega^n=\int_X e^v\omega^n>0,
$$
and
$$
\lim_{k\to\infty}\int_X Me^{v_{\varepsilon_k}}\,\omega^n=\int_X Me^v\omega^n>0,
$$
since $M$ is non negative and strictly positive in at least one point, while the set of points where $v=-\infty$ has zero measure.

Plugging this information into inequality (\ref{eq:final}) we obtain the desired contradiction since the left hand side blows up while the right hand side converges to some fixed positive number.

\bibliography{bibliography}{}

\end{document}